\documentclass[%
  a4paper,
  onecolumn,
  colorlinks,
]{preprint}

\usepackage{preamble}

\theoremstyle{plain}\newtheorem{theorem}{Theorem}

\begin{document}


\title{Using \texorpdfstring{$\boldsymbol{LDL^{\trans}}$}{LDL\_T}
  factorizations in Newton's method for solving general large-scale algebraic
  Riccati equations}

\author[$\ast$]{Jens Saak}
\affil[$\ast$]{Max Planck Institute for Dynamics of Complex Technical
  Systems, Sandtorstra{\ss}e 1, 39106 Magdeburg, Germany.\authorcr%
  \email{saak@mpi-magdeburg.mpg.de}, \orcid{0000-0001-5567-9637}}

\author[$\dagger$]{Steffen W. R. Werner}
\affil[$\dagger$]{Department of Mathematics and
    Division of Computational Modeling and Data Analytics,
    Academy of Data Science, Virginia Tech,
    Blacksburg, VA 24061, USA.\authorcr%
  \email{steffen.werner@vt.edu}, \orcid{0000-0003-1667-4862}}

\shorttitle{\texorpdfstring{$LDL^{\trans}$}{LDL\_T} factorizations
  for solving general CAREs}
\shortauthor{J. Saak, S.~W.~R. Werner}
\shortdate{2024-08-30}
\shortinstitute{}



\abstract{%
 Continuous-time algebraic Riccati equations can be found in many disciplines
in different forms.
In the case of small-scale dense coefficient matrices, stabilizing solutions
can be computed to all possible formulations of the Riccati equation.
This is not the case when it comes to large-scale sparse coefficient matrices.
In this paper, we provide a reformulation of the Newton-Kleinman iteration
scheme for continuous-time algebraic Riccati equations using indefinite
symmetric low-rank factorizations.
This allows the application of the method to the case of
general large-scale sparse coefficient matrices.
We provide convergence results for several prominent realizations of the
equation and show in numerical examples the effectiveness of the approach.

}

\novelty{%
  In this work, we use indefinite symmetric low-rank factorizations to
  extend the Newton-Kleinman iteration to general Riccati equations with
  large-scale sparse coefficient matrices.
  We provide a convergence theory for several prominent realizations of the
  equation and investigate different scenarios numerically.
}

\maketitle


\section{Introduction}%
\label{sec:intro}

The solutions to continuous-time algebraic Riccati equations (CAREs) play
essential roles for many concepts in systems and control theory.
They occur, for example, in the design of optimal and robust regulators for
dynamical processes~\cite{AndM90, LanR95, Loc01, Son98},
model order reduction methods for dynamical systems~\cite{BenS17, DesP82,
JonS83, OpdJ88},
network analysis~\cite{AndV72} or
applications with differential games~\cite{BasM17, Del07}.
In general, CAREs are quadratic matrix equations of the form
\begin{equation} \label{eqn:riccati}
  A^{\trans} X E + E^{\trans} X A + C^{\trans} Q C
    - {\left( B^{\trans} X E + S^{\trans} \right)}^{\trans} R^{-1}
    \left( B^{\trans} X E + S^{\trans} \right) = 0,
\end{equation}
with $A, E \in \R^{n \times n}$,
$B, S \in \R^{n \times m}$,
$C \in \R^{p \times n}$,
$Q = Q^{\trans} \in \R^{p \times p}$ and
$R = R^{\trans} \in \R^{m \times m}$ invertible.
For simplicity of illustration, we present the proposed algorithm and results
for the case that $E$ is invertible;
however, we outline modifications for the case of non-invertible $E$ matrices
in \Cref{sec:daes}.
Among all the solutions to~\cref{eqn:riccati}, the one of particular interest
in most cases is the stabilizing solution, here denoted as
$X_{\ast} \in \R^{n \times n}$, for which it holds that the eigenvalues of the
generalized matrix pencil
\begin{equation*}
  \lambda E - (A - B R^{-1} (B^{\trans} X_{\ast} E + S^{\trans}))
\end{equation*}
lie in the open left complex half-plane.
Matrix pencils with such an eigenvalue structure are also referred to as
Hurwitz.

In the case of dense coefficient matrices of small dimension $n$, a variety
of different approaches has been established for the numerical solution
of~\cref{eqn:riccati}.
Direct methods can be used to construct the solution via an eigenvalue
decomposition of the underlying Hamiltonian or even matrix
pencils~\cite{AmmBM93, ArnL84, Lau79}.
On the other hand, iterative approaches such as the matrix sign function
iteration and structure-preserving doubling avoid the eigendecomposition and
aim directly for the computation of the eigenspaces of
interest~\cite{BenEQetal14, ChuFL05, Kim89, Rob80}.
Other iterative approaches construct sequences of matrices that
converge to the stabilizing solution~\cite{Kle68, LanFAetal08, San74}.

With the problem dimension $n$ increasing, the task of
solving~\cref{eqn:riccati} becomes more complicated.
Even if in those cases $A$ and $E$ typically become sparse, the stabilizing
solution $X_{\ast}$ of~\cref{eqn:riccati} must be expected to
be densely populated.
Thus, the demands on computational resources such as time and memory become
infeasible when computing $X_{\ast}$ via classical approaches for
$n \in \mathcal{O}(10^{5})$ and larger.
Under the assumption that the dimensions of the factored coefficients
in~\cref{eqn:riccati} are significantly smaller than the solution dimension,
i.e., $p, m \ll n$, new iterative approaches for the solution
of~\cref{eqn:riccati} have been developed for some particular
realizations.
The key ingredient in all instances is the use of low-rank factorized
approximations of the solution $X_{\ast}$, typically given as
$Z_{\ast} Z_{\ast}^{\trans} \approx X_{\ast}$, where
$Z_{\ast} \in \R^{n \times \ell}$ and $\ell \ll n$.
This is justified by a fast singular value decay of the
solutions~\cite{BenB16, Sti18}.

For the special case that $S = 0$, $Q$ is symmetric positive semi-definite and
$R$ is symmetric positive definite, a variety of new approaches has been
developed in recent years.
Methods like the Newton and Newton-Kleinman iterations have been
extended~\cite{BenHSetal16, BenLP08} employing yet another low-rank solver such
as the low-rank alternating direction implicit (LR-ADI) method~\cite{BenKS13,
BenKS14, BenLP08, LiW02} for the Lyapunov equations occurring in every Newton
step.
Projection-based methods construct approximating subspaces to project the
coefficients of~\cref{eqn:riccati} onto smaller dimension and then solve
small-scale Riccati equations with classical dense
approaches~\cite{HeyJ09, Sim16}.
The Riccati alternating direction implicit (RADI)
method~\cite{BenBKetal18, BerF24} and the incremental low-rank subspace
iteration (ILRSI)~\cite{LinS15} are among the most successful low-rank solvers
for this variant of the Riccati equation.
We refer the reader to~\cite{BenBKetal20, BenS13, Kue16} for general overviews
and numerical comparisons of these methods.

In other instances of~\cref{eqn:riccati}, the amount of established methods
decreases significantly.
In the case of $Q$ symmetric positive semi-definite and $R$ symmetric negative
definite, only extensions of the Newton and Newton-Kleinman iteration have been
proposed for large-scale sparse systems~\cite{BenS14}.
Recently, a new low-rank method has been developed in~\cite{BenHW23} that
allows to compute the solution to~\cref{eqn:riccati} with indefinite $R$
and $Q$ symmetric positive semi-definite matrices.
Under the assumption that the stabilizing solution $X_{\ast}$ is symmetric
positive semi-definite, this new low-rank method iteratively approximates
$X_{\ast}$ via accumulating solutions to classical Riccati equations with
positive definite $R$ terms.

In this work, we are lifting all restrictions and investigate the numerical
approximation of the stabilizing solution to the general
CARE~\cref{eqn:riccati}.
Therefore, we focus on the Newton-Kleinman method~\cite{Kle68} and extend this
approach to the case of large-scale sparse coefficient matrices by utilizing an
indefinite symmetric low-rank factorization of the stabilizing solution.
We show that this new approach generalizes existing methods and provide a
theoretical background for several of the practically occurring scenarios.
The theoretical analysis is supported by multiple numerical experiments.

Throughout this paper, $A^{\trans}$ denotes the transpose of the matrix $A$.
Also, we denote symmetric positive (semi-)definite matrices $A$ 
by $A > 0$ ($A \geq 0$) and we write $A > B$ ($A \geq B$) if $A - B$ is
symmetric positive (semi-)definite. 
Similarly, we use $A < 0$ ($A \leq 0$) to denote symmetric negative
(semi-)definite matrices and write $A < B$ ($A \leq B$) if $A - B$ is 
symmetric negative (semi-)definite. 
Moreover, $\inner{., .}$ denotes the Frobenius inner product, i.e.,
$\inner{A, B} = \trace{A^{\trans}B}$ for real matrices $A$ and $B$ of compatible
dimensions.
By $I_{n}$ we denote the $n$-dimensional identity matrix.

The remainder of this paper is organized as follows:
In \Cref{sec:eqns}, we provide an overview about different realizations
of the continuous-time algebraic Riccati equation from the literature with their
motivational background and how they fit into the presented general
formulation~\cref{eqn:riccati}.
In \Cref{sec:newtonkleinman}, we derive the Newton-Kleinman formulation
for~\cref{eqn:riccati} based on which we extend the approach to the large-scale
sparse setting.
Afterwards, we provide a theoretical analysis of the convergence behavior,
formulas for an exact line search procedure in the Newton iteration and an
extension of the method to non-invertible $E$ matrices.
Numerical experiments to support the theoretical discussions of this paper
are conducted in \Cref{sec:numerics}.
The paper is concluded in \Cref{sec:conclusions}.


\section{Example equations from the literature}%
\label{sec:eqns}

Several realizations of CAREs are displayed throughout the literature.
The form~\cref{eqn:riccati} we consider in this work appears to be the most
general formulation of the CARE with factorized terms that allow for low-rank
approximations in the large-scale sparse setting.
Some of the most prominent realizations are outlined in the following.
These will also serve as examples in the numerical experiments in
\Cref{sec:numerics}.


\subsection{Linear-quadratic regulator problems}

First, we may consider the CARE formulation given in~\cref{eqn:riccati}.
With the additional assumptions that $Q \geq 0$ and $R > 0$, this realization
can be found in optimal control for the construction of optimal state-feedback
regulators~\cite{AndM90, Loc01, Son98}.
The corresponding optimization problem is given by
\begin{subequations}\label{eqn:optcontrol}
\begin{align} \label{eqn:optcontrol_a}
  & \min\limits_{u~\text{stab.}}
    \int\limits_{0}^{\infty} {y(t)}^{\trans} Q y(t)
    + {x(t)}^{\trans} S u(t) + {u(t)}^{\trans} R u(t) \operatorname{d}t\\
  \nonumber
  & \text{subject to}\\ \label{eqn:optcontrol_b}
  & \hspace{\baselineskip}
    \begin{aligned}
      E \dot{x}(t) & = A x(t) + B u(t), \\
      y(t) & = C x(t).
    \end{aligned}
\end{align}
\end{subequations}
The task is to find a controller $u$ that solves the optimization
problem~\cref{eqn:optcontrol_a} while stabilizing the corresponding dynamical
system~\cref{eqn:optcontrol_b}.
Assume that a stabilizing solution $X_{\ast}$ to~\cref{eqn:riccati} exists, then
the solution to~\cref{eqn:optcontrol} is given by
$u(t) = K_{\ast} x(t)$, where the feedback matrix is given by
$K_{\ast} = R^{-1} (B^{\trans} X_{\ast} E + S^{\trans})$.
If the matrix pencil $\lambda E - A$ is stabilizable with respect to $B$ and
observable with respect to $C$, then a sufficient condition for the existence
of the stabilizing solution $X_{\ast}$ is that
\begin{equation*}
  \begin{bmatrix} C^{\trans} Q C & S \\ S^{\trans} & R \end{bmatrix} \geq 0
\end{equation*}
holds; see~\cite{LanR95}.
Note that under the assumptions above, the stabilizing solution $X_{\ast}$ is
known to be positive semi-definite.


\subsection{Linear-quadratic-Gaussian control and unstable model order
  reduction}

A different realization of~\cref{eqn:riccati} relates to the construction of
optimal controllers and model order reduction of unstable dynamical systems.
Consider the modified optimal regulator problem
\begin{align*}
  & \min\limits_{u~\text{stab.}}
    \int\limits_{0}^{\infty} {y(t)}^{\trans} \tQ y(t)
    + {u(t)}^{\trans} \tR u(t) \operatorname{d}t\\
  & \text{subject to}\\
  & \hspace{\baselineskip}
    \begin{aligned}
      E \dot{x}(t) & = A x(t) + B u(t), \\
      y(t) & = C x(t) + D u(t),
    \end{aligned}
\end{align*}
with the feed-through matrix $D \in \R^{p \times m}$,
$\tQ \geq 0$ and $\tR > 0$.
The corresponding CARE, whose stabilizing solution provides the
optimal stabilizing control, is given by
\begin{equation} \label{eqn:lqgcare}
  A^{\trans} X E + E^{\trans} X A + C^{\trans} \tQ C -
    {\left(B^{\trans} X E + D^{\trans} C\right)}^{\trans}
    {\left( \tR + D^{\trans} D \right)}^{-1}
    \left(B^{\trans} X E + D^{\trans} C\right) = 0.
\end{equation}
The equation~\cref{eqn:lqgcare} can be rewritten as~\cref{eqn:riccati} by
setting
\begin{equation*}
  Q = \tQ, \quad
  R = \tR + D^{\trans} D, \quad
  S = D^{\trans} C.
\end{equation*}
The very same equation~\cref{eqn:lqgcare} can also be found in the design of
optimal linear-quadratic Gaussian (LQG) controllers and in the LQG
balanced truncation method that is used for the computation of reduced-order
dynamical systems with unstable dynamics~\cite{BenS17, JonS83}.


\subsection{\texorpdfstring{$\boldsymbol{\Hinf}$}{H-infinity} control and
  robust model order reduction}

Another realization related to controller design and model order reduction comes
in the form of the $\Hinf$-Riccati equation
\begin{equation} \label{eqn:hinfcare}
  A^{\trans} X E + E^{\trans} X A + C^{\trans} \tQ C - E^{\trans} X \left(
    B_{2} \tR^{-1} B_{2}^{\trans} -
    \frac{1}{\gamma^{2}} B_{1} B_{1}^{\trans}
    \right) X E = 0;
\end{equation}
see~\cite{BenHW22, BenHW23, MusG91}.
This equation is typically associated with dynamical systems of the form
\begin{align*}
  E \dot{x}(t) & = A x(t) + B_{1} w(t) + B_{2} u(t), \\
  y(t) & = C x(t),
\end{align*}
where $B_{1} \in \R^{n \times m_{1}}$ models the influence of external
disturbances on the control problem and $B_{2} \in \R^{n \times m_{2}}$ are the
actual control inputs.
The dimensions of $B_{1}$ and $B_{2}$ are related to~\cref{eqn:riccati} via
$m = m_{1} + m_{2}$.
The matrices $\tR > 0$ and $\tQ \geq 0$ are weighting matrices from the
associated optimal control problem similar to~\cref{eqn:optcontrol_a},
and $\gamma > 0$ is the robustness margin that is achieved by the constructed
regulator/controller.
Equations of the form~\cref{eqn:hinfcare} can be rewritten
into~\cref{eqn:riccati} via
\begin{equation*}
  B = \begin{bmatrix} B_{1} & B_{2} \end{bmatrix}, \quad
  Q = \tQ, \quad
  R = \begin{bmatrix} -\gamma^{2} I_{m_{1}} & 0 \\ 0 & \tR \end{bmatrix}, \quad
  S = 0.
\end{equation*}
In this case, the quadratic weighting term $R$ in~\cref{eqn:riccati} is
indefinite.
As in the previous examples, one is interested in the stabilizing solution
$X_{\ast}$ to~\cref{eqn:hinfcare}, which might be indefinite, here, due to the
indefiniteness of $R$.


\subsection{Passivity, contractivity and spectral factorizations}

As last examples, we would like to mention two equations that are related
to dynamical system properties such as contractivity and passivity as well as
spectral factorizations of rational
functions~\cite{BenS17, DesP82, Gre88, OpdJ88}.
The so-called bounded-real Riccati equation is given as
\begin{equation} \label{eqn:brcare}
  A^{\trans} X E + E^{\trans} X A + C^{\trans} C +
    {\left( B^{\trans} X E + D^{\trans} C \right)}^{\trans}
    {\left( \gamma^{2} I_{m} - D^{\trans} D \right)}^{-1}
    \left( B^{\trans} X E + D^{\trans} C \right) = 0,
\end{equation}
with $D \in \R^{p \times m}$ and $\gamma > \lVert H \rVert_{\Hinf}$, where
$\lVert . \rVert_{\Hinf}$ denotes the $\Hinf$ Hardy norm~\cite{Sty06} and
$H(s) = C {(s E - A)}^{-1} B + D$ is a rational function in the complex
variable $s \in \C$.
On the other hand, the positive-real Riccati equation reads
\begin{equation} \label{eqn:prcare}
  A^{\trans} X E + E^{\trans} X A + {\left( B^{\trans} X E - C \right)}^{\trans}
    {\left( D^{\trans} + D \right)}^{-1}
    \left( B^{\trans} X E - C \right) = 0,
\end{equation}
where the dimensions satisfy $m = p$.
Equation~\cref{eqn:brcare} can be rewritten as~\cref{eqn:riccati} by choosing
\begin{equation*}
  Q = I_{p}, \quad
  R = -\left( \gamma^{2} I_{m} - D^{\trans} D \right), \quad
  S = C^{\trans} D,
\end{equation*}
and equation~\cref{eqn:prcare} can be reformulated using
\begin{equation*}
  Q = 0, \quad
  R = -\left( D^{\trans} + D \right), \quad
  S = C^{\trans}.
\end{equation*}
With the assumptions above, the $R$ matrix is symmetric negative definite in
both cases, while $Q$ is either symmetric positive definite or $0$.
Again only the stabilizing solutions $X_{\ast}$ to~\cref{eqn:brcare}
or~\cref{eqn:prcare} are of any interest.
Despite the changed definiteness of $R$ here, the stabilizing solutions are
positive semi-definite~\cite{BenS14}.


\section{Low-rank inexact Newton-Kleinman iteration with line search}%
\label{sec:newtonkleinman}

In this section, we derive the low-rank Newton-Kleinman iteration and provide
the formulas for inexact steps, a line search approach and projected Riccati
equations.


\subsection{Derivation of the low-rank Newton-Kleinman scheme}%
\label{sec:derivation}

Solving the CARE~\cref{eqn:riccati} is a root-finding problem with a nonlinear
matrix-valued equation and solution.
Therefore, Newton's method is a valid approach to compute a solution to the
problem~\cite{San74}, and it has been shown in many cases that the computed
solution is the desired stabilizing one.
The basic method can be derived by considering the Fr{\'e}chet derivative of
the Riccati operator
\begin{equation} \label{eqn:careoper}
  \cR(X) = A^{\trans} X E + E^{\trans} X A + C^{\trans} Q C
      - {\left( B^{\trans} X E + S^{\trans} \right)}^{\trans} R^{-1}
      \left( B^{\trans} X E + S^{\trans} \right),
\end{equation}
with respect to the unknown $X$.
The first Fr{\'e}chet derivative of~\cref{eqn:careoper} with respect to $X$ and
evaluated in $N$ is given by
\begin{equation*}
  \cR'(X)(N) = {\left( A - B R^{-1} (B^{\trans} X E + S^{\trans})
    \right)}^{\trans} N E + E^{\trans} N {\left( A - B R^{-1} (B^{\trans} X E +
    S^{\trans}) \right)},
\end{equation*}
and the second Fr{\'e}chet derivative with respect to $X$ evaluated in $N_{1}$
and $N_{2}$ is independent of $X$ and can be written as
\begin{equation*}
  \cR''(X)(N_{1}, N_{2}) = -E^{\trans} N_{1} B R^{-1} B^{\trans} N_{2} E -
    E^{\trans} N_{2} B R^{-1} B^{\trans} N_{1} E.
\end{equation*}
As outlined in~\cite{BenHSetal16}, the classical Newton approach is
usually undesired in the case of large-scale sparse coefficients when
compared to the reformulation given by the Newton-Kleinman scheme~\cite{Kle68}.
Either method is based on the solution of a Lyapunov equation in every iteration
step.
However, while the classical Newton method computes an update to the current
iterate of the form $X_{k+1} = X_{k} + N_{k}$, where $N_{k}$ is given as the
solution to a Lyapunov equation, the Newton-Kleinman method computes $X_{k+1}$
directly as the solution of the Lyapunov equation that is given by
\begin{equation*}
  \cR'(X_{k})(X_{k+1}) = \cR'(X_{k})(X_{k}) - \cR(X_{k});
\end{equation*}
see, for example,~\cite{BenHSetal16}.
Therefore, the Newton-Kleinman approach for~\cref{eqn:riccati} is given by
solving the following Lyapunov equation
\begin{equation} \label{eqn:lyap_old}
  A_{k}^{\trans} X_{k + 1} E + E^{\trans} X_{k + 1} A_{k}
    + C^{\trans} Q C + K_{k}^{\trans} R K_{k} - S K_{k} - {(S K_{k})}^{\trans}
    = 0
\end{equation}
in every iteration step, with $A_{k} = A - B K_{k}$ and
$K_{k} = R^{-1}(B^{\trans} X_{k} E + S^{\trans})$,
and starting with some initial stabilizing feedback $K_{0}$;
see~\cite{Arn84, ArnL84}.
This $K_{0}$ is chosen such that all eigenvalues of
$\lambda E - (A - B K_{0})$ lie in the open left half-plane.

To extend the scheme~\cref{eqn:lyap_old} to the large-scale sparse setting,
we must first observe that the part of the equation that does not contain
the current unknown $X_{k+1}$, in other words, the constant term is an
indefinite symmetric matrix.
To utilize this form of the constant term, similar to the argumentation
in~\cite{LanMS14}, we propose to approximate the solution matrix
to~\cref{eqn:riccati} by a symmetric indefinite low-rank factorization of the
form
\begin{equation} \label{eqn:ldlt}
  L D L^{\trans} \approx X,
\end{equation}
where $L \in \R^{n \times \ell}$ and $D \in \R^{\ell \times \ell}$ symmetric.
By the low-rank structure of the constant term coefficient matrices
of~\cref{eqn:riccati}, as well as its quadratic terms  (since $m, p \ll n$),
we expect the solution to have numerically a low rank such that $\ell \ll n$
holds~\cite{BenB16,Sti18}.
Rewriting the constant term of~\cref{eqn:lyap_old} into the same shape
as~\cref{eqn:ldlt} yields
\begin{align} \nonumber
  & C^{\trans} Q C + K_{k}^{\trans} R K_{k} - S K_{k} - {(S K_{k})}^{\trans}\\
  \label{eqn:ldltrhs_old}
  & = \begin{bmatrix} C^{\trans} & K_{k}^{\trans} & S \end{bmatrix}
    \begin{bmatrix}
      Q & 0 & 0 \\ 0 & R & -I_{m} \\ 0 & -I_{m} & 0
    \end{bmatrix}
    \begin{bmatrix} C \\ K_{k} \\ S^{\trans} \end{bmatrix},
\end{align}
where the center matrix has the dimension $2m + p$.
It is possible to avoid the switching term (the two negative identities) in the
lower right corner of the center matrix in~\cref{eqn:ldltrhs_old} by making the
following reformulations:
\begin{align} \nonumber
  & C^{\trans} Q C + K_{k}^{\trans} R K_{k} - S K_{k} - {(S K_{k})}^{\trans} \\
  \nonumber
  & = C^{\trans} Q C + K_{k}^{\trans} R K_{k} - S K_{k} - {(S K_{k})}^{\trans}
    - \underbrace{S R^{-1} S^{\trans} + S R^{-1} S^{\trans}}_{=\,0}\\
  \nonumber
  & =  C^{\trans} Q C - S R^{-1} S^{\trans}
    + {(K_{k} - R^{-1} S^{\trans})}^{\trans} R (K_{k} - R^{-1} S^{\trans}) \\
  \label{eqn:ldltrhs}
  & = \begin{bmatrix} C^{\trans} & (R^{-1} S^{\trans})^{\trans} &
    {(K_{k} - R^{-1} S^{\trans})}^{\trans} \end{bmatrix}
    \begin{bmatrix} Q & 0 & 0 \\ 0 & -R & 0 \\ 0 & 0 & R \end{bmatrix}
    \begin{bmatrix} C \\ R^{-1} S^{\trans} \\ K_{k} - R^{-1} S^{\trans}
    \end{bmatrix}.
\end{align}
The reformulation in~\cref{eqn:ldltrhs} also has $2m + p$ as inner dimension
of the factors and features a block diagonal center matrix, which we believe to
be advantageous in the implementation.
Plugging~\cref{eqn:ldltrhs} into~\cref{eqn:lyap_old} yields the final
$LDL^{\trans}$-factorized Lyapunov equation that we employ in our new
Newton-Kleinman iteration.
The resulting method is summarized in \Cref{alg:newton}.
Lyapunov equations such as in \Cref{alg:newton_lyap} of \Cref{alg:newton} can
be efficiently solved, for example, via the $LDL^{\trans}$-factorized
low-rank ADI method in~\cite{LanMS14}.

\begin{algorithm}[t]
  \SetAlgoHangIndent{1pt}
  \DontPrintSemicolon%
  \caption{$LDL^{\trans}$-factorized low-rank Newton-Kleinman iteration.}%
  \label{alg:newton}

  \KwIn{Matrices $A, B, S, C, Q, R, E$ from~\cref{eqn:riccati}, stabilizing
    feedback $K_{0}$ such that $\lambda E - (A - B K_{0})$ is Hurwitz.}
  \KwOut{Approximation $L_{k} D_{k} L_{k}^{\trans} \approx X_{\ast}$ to the
    stabilizing solution of~\cref{eqn:riccati}.}

  Initialize $k = 0$,
    \begin{equation*}
      V^{\trans} = R^{-1} S^{\trans} \quad \text{and} \quad
      T = \begin{bmatrix} Q & 0 & 0 \\ 0 & -R & 0 \\ 0 & 0 & R
        \end{bmatrix}.
    \end{equation*}\;
    \vspace{-\baselineskip}

  \While{not converged}{
     Construct the residual term
      \begin{equation*}
        W_{k} = \begin{bmatrix} C \\ V^{\trans} \\ K_{k} - V^{\trans}
          \end{bmatrix}.
      \end{equation*}
      \vspace{-\baselineskip}\;

    Solve the Lyapunov equation
      \begin{equation*}
        A_{k}^{\trans} X_{k+1} E + E^{\trans} X_{k+1} A_{k} +
          W_{k}^{\trans} T W_{k} = 0,
      \end{equation*}
      for $L_{k+1} D_{k+1} L_{k+1}^{\trans} \approx X_{k+1}$ and where
      $A_{k} = A - B K_{k}$.\;%
      \label{alg:newton_lyap}

    Update the feedback matrix \(K_{k+1} = R^{-1} \left(B^{\trans}
      (L_{k + 1} D_{k + 1} L_{k + 1}^{\trans}) E + S^{\trans}\right)\).\;

    Increment $k \gets k + 1$.
  }
\end{algorithm}


\subsection{An equivalent reformulation via low-rank updates}

The efficient handling of the low-rank-updated operator
$A_{k} = A - B K_{k}$ in the Lyapunov equation in \Cref{alg:newton_lyap}
of \Cref{alg:newton} is essential for computing its solution in the large-scale
sparse case.
Typically, linear systems of equations with $A_{k}$ need to be solved, which
can be effectively implemented using the Sherman-Morrison-Woodbury matrix
inversion formula or the augmented matrix approach; see, for
example,~\cite{GolV13}.
Since the handling of such operators is already implemented in most software
for the solution of matrix equations such as the M-M.E.S.S.
library~\cite{BenKS21}, we may use a reformulation of~\cref{eqn:riccati} to
hide the $S$ inside the other matrices.
First, we observe that by multiplying out the terms in~\cref{eqn:riccati}, we
obtain the equivalent CARE
\begin{align} \nonumber
  & {\left( A - B R^{-1} S^{\trans} \right)}^{\trans} X E +
    E^{\trans} X \left( A - B R^{-1} S^{\trans} \right) +
    \left( C^{\trans} Q C - S R^{-1} S^{\trans} \right)\\
  \label{eqn:riccati_tmp}
  & \quad{}-{} E^{\trans} X B R^{-1} B^{\trans} X E = 0.
\end{align}
After some renaming of the terms in~\cref{eqn:riccati_tmp}, we obtain
\begin{equation} \label{eqn:riccati2}
  \hA^{\trans} X E + E^{\trans} X \hA + \hC^{\trans} \hQ \hC
    - E^{\trans} X B R^{-1} B^{\trans} X E = 0,
\end{equation}
where
\begin{equation} \label{eqn:extcoef}
  \hA = A + U V^{\trans}, \quad
  U = -B, \quad
  V = S R^{-\trans}, \quad
  \hC = \begin{bmatrix} C \\ S^{\trans} \end{bmatrix}, \quad
  \hQ = \begin{bmatrix} Q & 0 \\ 0 & -R^{-1} \end{bmatrix}.
\end{equation}
Running \Cref{alg:newton} for the renamed matrices $\hA$, $B$, $\hS = 0$, $\hC$,
$\hQ$, $R$ and $E$ will yield exactly the same iterates computed in every step,
while hiding the original $S$ term in $\hA$ and $\hC$. Note here that all the
corresponding stabilizing feedbacks $\hK_{k}$ are changed such that
$\lambda E - (\hA - B\hK_{k})$ is stabilized rather than
$\lambda E - (A - B K_{k})$.  In particular, the initial stabilizing feedback
must be chosen correctly.  On the other hand, the final stabilizing feedback
$\hK_{k_{\max}}$, corresponding to the final iterate $\hX_{k_{\max}}$, can
easily be modified to stabilize the true associated matrix pencil
$\lambda E - A$ via
\begin{equation*}
  K_{k_{\max}} = \hK_{k_{\max}} + V^{\trans}.
\end{equation*}

While the two formulations~\cref{eqn:riccati,eqn:riccati2} are
equivalent, employing the Newton-Kleinman method for~\cref{eqn:riccati2}
rather than~\cref{eqn:riccati} is expected to be mildly more expensive in the
general case because the dimension of the constant term stays unchanged
while the dimension of the low-rank updates in $A_{k}$ are increased by $m$
columns.
More precisely, when implementing the Newton-Kleinman method for~\cref{eqn:riccati2}, the matrix $A_{k} = A - B K_{k}$ in \Cref{alg:newton}
changes to
\begin{equation*}
  \hA_{k} = A + U V^{\trans} - B \hK_{k} = A +\begin{bmatrix} U & -B
    \end{bmatrix} \begin{bmatrix}  V & \hK_{k} \end{bmatrix}^{\trans}.
\end{equation*}
It now has two low-rank updates, with \(m\) columns each, that can be rewritten
into one low-rank update with $2m$ columns, instead of only $m$ columns, as
before.
This $\hA_{k}$ can be handled similarly to the original $A_{k}$ in the solver
of the Lyapunov equation without explicitly forming the dense matrix $\hA_{k}$;
however this step becomes more expensive since the low-rank update has more
columns.
On the other hand, considering the example equations from \Cref{sec:eqns}, we
can also see a reduction in computational costs using~\cref{eqn:riccati2}.
In
\makeatletter
\@ifclassloaded{preprint}{\Cref{eqn:lqgcare,eqn:brcare,eqn:prcare},}{%
equations~\cref{eqn:lqgcare,eqn:brcare,eqn:prcare},}
\makeatother
the $S$ term is a multiplication of the constant term $C$ with some
appropriately sized matrix $D$.
This fact allows us some additional dimension reduction of the constant term
in~\cref{eqn:riccati2}.
As example, consider the case of the LQG CARE in~\cref{eqn:lqgcare}:
The constant term in~\cref{eqn:riccati2} can be written as

\begin{equation*}
  \resizebox{.99\textwidth}{!}{%
  \ensuremath{\hC^{\trans}\hQ\hC=
    C^{\trans}\tQ C - C^{\trans} D {(\tR + D^{\trans} D)}^{-1} D^{\trans} C
  = C^{\trans}\underbrace{\left(\tQ - D {(\tR + D^{\trans} D)}^{-1}
    D^{\trans}\right)}_{= \check{Q}} C
  = C^{\trans} \check{Q} C.}}
\end{equation*}

In such cases, the size of the constant term in~\cref{eqn:riccati2} can be
reduced from $2m + p$ to $m + p$, which improves the performance of large-scale
sparse solvers that build the solution using the constant term.
We have implemented this version of \Cref{alg:newton} for the reformulated
CARE~\cref{eqn:riccati2} in the M-M.E.S.S. library~\cite{BenKS21} for our
numerical experiments due to the easy integration into the existing framework
of CAREs of the form~\cref{eqn:riccati2}.

\begin{remark}
  Projection-type methods such as the extended and rational Krylov subspace
  methods~\cite{HeyJ09, Sim16} naturally construct the solution to Riccati
  equations~\cref{eqn:riccati} in the same symmetric indefinite low-rank
  factorized form~\cref{eqn:ldlt} that we propose here for the Newton-Kleinman
  method.
  Let $\cV \in \R^{n \times r}$ be a basis matrix of a suitable $r$-dimensional
  projection space, in projection methods, the stabilizing solution
  to~\cref{eqn:riccati} is approximated via
  $\cV \bX_{\ast} \cV^{\trans} \approx X_{\ast}$, where
  $\bX_{\ast} \in \R^{r \ \times r}$ is the stabilizing solution to
  the projected Riccati equation
  \begin{equation} \label{eqn:projriccati}
    \bA^{\trans} \bX \bE + \bE^{\trans} \bX \bA + \bC^{\trans} Q \bC
      - {\left( \bB^{\trans} X \bE + \bS^{\trans} \right)}^{\trans} R^{-1}
      \left( \bB^{\trans} X \bE + \bS^{\trans} \right) = 0,
  \end{equation}
  where $\bA = \cV^{\trans} A \cV \in \R^{r \times r}$,
  $\bE = \cV^{\trans} E \cV \in \R^{r \times r}$,
  $\bB = \cV^{\trans} B \in \R^{r \times m}$,
  $\bS = \cV^{\trans} S \in \R^{r \times m}$ and
  $\bC = C \cV \in \R^{p \times r}$.
  We note that in the case of Krylov subspace methods, which construct the
  projection space adaptively using the matrix pencil $\lambda E - A$ and the
  constant term, it is necessary to make use of the extended constant term
  $\begin{bmatrix} C^{\trans} & S \end{bmatrix}$ from~\cref{eqn:extcoef}
  rather than $C$ alone. Moreover, careful verification that the new
  definiteness assumptions do not break the residual formulas would be
  necessary.

  Besides the careful handling of basis extensions to ensure numerical
  stability, projection methods for Riccati equations are rather difficult to
  analyze in terms of their convergence behavior and, in particular, the
  solvability of the projected Riccati equation~\cref{eqn:projriccati}.
  This has been done only recently for the standard equation case
  ($S = 0, Q \geq 0, R > 0$) under rather
  restrictive assumptions in~\cite{ZhaFC20} and we will not consider these
  types of methods in the remainder of this manuscript.
\end{remark}

\begin{remark}
   While the previous remark was concerned with direct application of the
   projection method to the ARE, it would also be possible to integrate our
   proposed Newton loop with the extended Krylov subspace method (EKSM) along
   the lines of~\cite{Pal19}. Then, only projected Lyapunov equations need to be
   solved in every step, whose solvability is easier to decide. However,
   verification that the same recycling of the EKSM spaces across the Newton
   steps is still possible and careful implementation is beyond the scope of
   this manuscript.
   Further, the RADI method~\cite{BenBKetal18}, could be extended to
   AREs~(\ref{eqn:riccati}). Its equivalence to Krylov projection methods for
   the case ($S = 0, Q \geq 0, R > 0$) would need to be verified for the new
   situation. While first numerical experiments with RADI for \(S=0\)
   look promising, results need further research and will be reported
   separately.
\end{remark}

\subsection{Convergence results}

In the following theorem, we collect the convergence results for two distinct
cases of~\cref{eqn:riccati}, depending on the definiteness of the quadratic
term.
The results are formulated for the exact iterates $X_{k}$ of the
Newton-Kleinman iteration in \Cref{alg:newton} rather than the low-rank
approximations $L_{k} D_{k} L_{k}^{\trans}$ since the
approximation errors introduced by the numerical solution of the Lyapunov
equations and truncation of the solution factors to small numerical ranks may
render the results wrong.
However, for accurate enough approximations, the results remain true in
practice.

\begin{theorem}\label{thm:convergence}
  Assume~\cref{eqn:riccati} has a unique stabilizing solution
  $X_{\ast}$, let $K_{0}$ be a feedback matrix such that the eigenvalues of
  $\lambda E - (A - B K_{0})$ lie in the open left complex half-plane and
  let either $R > 0$ or $R < 0$ be true.
  Then, for the exact solutions $X_{k} = L_{k} D_{k} L_{k}^{\trans}$ to
  the Lyapunov equations in \Cref{alg:newton_lyap} of \Cref{alg:newton},
  it holds that
  \begin{itemize}
    \item[(i)] the closed-loop pencils $\lambda E - A_{k}$ with
      $A_{k} = A - B K_{k}$ are stable for all $k \geq 0$,
    \item[(ii)] $\lim\limits_{k \to \infty} X_{k} = X_{\ast}$ and
      $\lim\limits_{k \to \infty} \calR(X_{k}) = 0$,
    \item[(iii)] the iterates $X_{k}$ converge globally and quadratic to
      $X_{\ast}$,
    \item[(iv)] if $R > 0$, then
      \begin{equation*}
        X_{1} \geq \cdots \geq X_{k} \geq X_{k + 1} \geq \cdots \geq X_{\ast},
      \end{equation*}
      and if $R < 0$, then
      \begin{equation*}
        X_{1} \leq \cdots \leq X_{k} \leq X_{k + 1} \leq \cdots \leq X_{\ast}.
      \end{equation*}
  \end{itemize}
\end{theorem}
\begin{proof}
  The results have been proven for the case $R > 0$ in~\cite{Arn84}.
  In the case of $R < 0$, we may use the convergence results
  from~\cite[Thm.~3.2]{BenS14}, which are based on earlier results
  from~\cite{Ben97b, Var95b}.
  Therefore, we consider the equivalent reformulation of the CARE into
  the classical form~\cref{eqn:riccati2}.
  Since $R < 0$, it holds that $-R > 0$ and therefore, we may
  write~\cref{eqn:riccati2} as
  \begin{equation*}
    \hA^{\trans} X E + E^{\trans} X \hA + \hC^{\trans} \hQ \hC
    + E^{\trans} X B \hR^{-1} B^{\trans} X E = 0,
  \end{equation*}
  where $\hR = -R > 0$.
  With $\hK_{0} = K_{0} - R^{-1} S^{\trans}$, the matrix pencil
  $\lambda E - (\hA - B \hK_{0})$ is stable and since $\hR > 0$, we
  have that $B \hR^{-1} B^{\trans} \geq 0$.
  Thus, the assumptions of~\cite[Thm.~3.2]{BenS14} are satisfied, which proofs
  the results of this theorem.
\end{proof}

\Cref{thm:convergence} shows that the general convergence behavior of
\Cref{alg:newton} changes with the definiteness of the quadratic term.
The other terms $C$, $Q$ and $S$ only affect the definiteness of the stabilizing
solution $X_{\ast}$ to which the method converges.
Beyond the convergence theory, the reformulations made in \Cref{sec:derivation}
and in the proof of \Cref{thm:convergence} show that in exact arithmetic,
the proposed \Cref{alg:newton} provides exactly the same iterates as the
Newton-Kleinman methods developed in~\cite{Arn84, BenS14}.

The techniques used to show all the convergence results in the proofs
in~\cite{Arn84, BenS14} are based on the main observation that the difference of
two consecutive steps in the Newton-Kleinman scheme is given as the unique
solution of the Lyapunov equation
\begin{equation*}
  X_{k} - X_{k+1} = \int\limits_{0}^{\infty}
    {\left(e^{(A E^{-1} - B K_{k}) t} \right)}^{\trans}
    {(K_{k - 1}  - K_{k})}^{\trans} R (K_{k - 1}  - K_{k})
     e^{(A E^{-1} - B K_{k}) t} \operatorname{d}t.
\end{equation*}
The definiteness of the difference $X_{k} - X_{k+1}$ depends thereby on the
definiteness of the $R$ matrix resulting in the monotonic convergence behavior
described in \Cref{thm:convergence}.
In the case of indefinite $R$, this monotonic behavior is likely to be lost as
we will demonstrate later in \Cref{sec:numerics}.
However, we were not able to construct a case for which a stabilizing solution
$X_{\ast}$ exists while the Newton-Kleinman method (\Cref{alg:newton}) diverges
or converges to the wrong solution.
In fact, \Cref{alg:newton} only diverged in experiments in which there was no
stabilizing  $X_{\ast}$.
This indicates that \Cref{alg:newton} may always converge to the correct
solution; however, most of the convergence results in \Cref{thm:convergence}
will not be true anymore for the case of indefinite $R$.

A different approach that provides a convergence theory for the case of
indefinite $R$ is the Riccati iteration~\cite{BenHW23, LanFAetal08}.
This method has been designed to solve Riccati equations with
indefinite quadratic terms of the form
\begin{equation} \label{eqn:riccati_indef}
  A^{\trans} X E + E^{\trans} X A + C^{\trans} C - E^{\trans} X \left(
    B_{2} B_{2}^{\trans} - B_{1} B_{1}^{\trans} \right) X E = 0.
\end{equation}
Note that this Riccati equation~\cref{eqn:riccati_indef} is contained in the
general case~\cref{eqn:riccati} that we consider here.
The method iterates on stabilizing solutions to Riccati equations with
positive semi-definite quadratic terms of the form
\begin{equation} \label{eqn:ri_care}
  A_{k}^{\trans} N_{k + 1} E +
    E^{\trans} N_{k + 1} A_{k}
    - E^{\trans} N_{k + 1} B_{2} B_{2}^{\trans}
    N_{k + 1} E
    + (E^{\trans} N_{k} B_{1})
    (E^{\trans} N_{k} B_{1})^{\trans} = 0,
\end{equation}
with $A_{k} = A - (B_{2} B_{2}^{\trans} - B_{1} B_{1}^{\trans})
X_{k}^{\operatorname{RI}} E$ and where the iterates are then given via
accumulation such that
$X_{k + 1}^{\operatorname{RI}} = X_{k}^{\operatorname{RI}} + N_{k + 1}$.
For the complete Riccati iteration, including the first initialization step,
implementational reformulations and convergence theory see~\cite{BenHW23}.
Under the additional assumption that $X_{\ast} \geq 0$ holds, the iterates
constructed by the Riccati iteration converge monotonically towards
$X_{\ast}$ as
\begin{equation*}
  X_{0}^{\operatorname{RI}} \leq \cdots \leq X_{k}^{\operatorname{RI}} \leq
    X_{k + 1}^{\operatorname{RI}} \leq X_{\ast}.
\end{equation*}
Since each of these iterates is computed via a CARE of the
form~\cref{eqn:ri_care} with $R > 0$, this overall iteration scheme can be
interpreted as splitting the two opposing convergence behaviors
for~\cref{eqn:riccati_indef} in \Cref{thm:convergence} into an inner and
an outer iteration.
Similar to the Newton methods, the Riccati iteration provides global quadratic
convergence.
The main difference to the results in \Cref{thm:convergence} is that the
closed-loop matrix pencils constructed in the outer loop of the iteration are
not guaranteed to be stable such that additional stabilization might be needed
to employ an inner CARE solver in the large-scale sparse case.


\subsection{Inexact Newton with exact line search}

Newton's method with exact line search has first been discussed for dense
generalized algebraic Riccati equations in~\cite{BenB98}.
Based on this work, Weichelt et al.~\cite{BenHSetal16, Wei16} formulated an
inexact low-rank Newton-ADI method with exact line search, focusing on the
representation of solutions in the form $X \approx Z Z^{\trans}$.
Since, in this work, we are pointing out advantages of the
$X\approx LDL^{\trans}$ representation, we provide the required formulas in
this context and show that they can also be evaluated at low cost.

To this end, we may call the $k$-th and $(k+1)$-st classic Newton-Kleinman
iterates $\Xold$ and $\Xnew$ and note that they are connected via the step
matrix $\Nk$, since $\Xnew = \Xold + \Nk$.
Further, we denote the $(k+1)$-st iterate after line search with the resulting
step size $\xi_{k}$ as $\Xlsk = \Xold + \xi_{k} \Nk$.
In~\cite[Chap.~6]{Wei16}, using earlier results from~\cite{Ben97,BenB98} for the
dense case, the
author shows that the dependence on the step size $\xi_{k}$ of the squared
Riccati residual norm, in the $k$-th Newton step,
forms a quartic polynomial
\begin{equation}\label{eqn:poly}
  \begin{aligned}
    f_{\cR,k}(\xi) & = \nrm[F]{\cR(\Xls)}^{2}\\
    & = {(1-\xi)}^{2}v_{1}^{(k)} + \xi^{2}v_{2}^{(k)} + \xi^{4} v_{3}^{(k)}
    + 2\xi(1-\xi)v_{4}^{(k)} - 2\xi^{2}(1-\xi)v_{5}^{(k)} - 2\xi^{3}v_{6}^{(k)}.
  \end{aligned}
\end{equation}
The coefficients are expressed in terms of the norms of the Riccati residual
and its derivatives evaluated in the above quantities, and expressed in low-rank
form.
In the context of the equations investigated here, these terms become
\begin{equation*}
  \begin{aligned}
    v_{1}^{(k)} & = \nrm[F]{\cR(\Xold)}^{2}
      = \trace{{(\Uk \Dk {\Uk }^{\trans})}^{2}}
      = \trace{{({\Uk }^{\trans}\Uk \Dk )}^{2}},\\
    v_{2^{(k)}} & = \nrm[F]{\cL(\Xnew)}^{2}
      = \trace{{(F_{k+1}G_{k+1}{F_{k+1}}^{\trans})}^{2}}
      = \trace{{({F_{k+1}}^{\trans}F_{k+1}G_{k+1})}^{2}},\\
    v_{3}^{(k)} & = \nrm[F]{\frac{1}{2}\cR''(\Xnew)(\Nk, \Nk)}^{2}
      = \nrm[F]{E^{\trans}\Nk BR^{-1}B^{\trans}\Nk E}^{2}\\
    & = \trace{{(\dKnew R\dKnew^{\trans})}^{2}}
      = \trace{{({\dKnew}^{\trans}\dKnew R)}^{2}},\\
    v_{4}^{(k)} & = \inner{\cR(\Xold),\cL(\Xnew)}
      = \trace{\Uk \Dk {\Uk }^{\trans} F_{k+1}G_{k+1}{F_{k+1}}^{\trans}}\\
    & = \trace{{F_{k+1}}^{\trans}\Uk \Dk \Uk^{\trans} F_{k+1}G_{k+1}},\\
    v_{5}^{(k)} & = \inner{\cR(\Xold),\cR''(\Xnew)(\Nk,\Nk)}
      = \trace{\Uk \Dk {\Uk }^{\trans} \dKnew R \dKnew^{\trans}}\\
    & = \trace{\dKnew^{\trans}\Uk \Dk \Uk^{\trans} \dKnew R},\\
    v_{6}^{(k)} & = \inner{\cL(\Xnew),\cR''(\Xnew)(\Nk,\Nk)}
      = \trace{F_{k+1}G_{k+1}{F_{k+1}}^{\trans} \dKnew R \dKnew^{\trans}}\\
    & = \trace{\dKnew^{\trans}F_{k+1}G_{k+1}{F_{k+1}}^{\trans} \dKnew R}.
  \end{aligned}
\end{equation*}
This is employing the Fr{\'e}chet derivatives from \Cref{sec:derivation}, and we
use that $\Nk = \Xnew - \Xold$. Consequently,
$R^{-1}B^{\trans}\Nk E = \Knew - \Kold = \dKnew$ holds.
Further, we have defined $\Uk = \begin{bmatrix} F_{k} & \dKold \end{bmatrix}$
and
\begin{equation*}
  \Dk = \begin{bmatrix} G_{k} & 0 \\ 0 & -R \end{bmatrix},
\end{equation*}
to express the Riccati residual in the \(k\)-th Newton step as
\(\cR(\Xold)=\Uk\Dk\Uk^{\trans}\), extending~\cite[Eqn.~(6.33b)]{Wei16} to
non-trivial center matrices.
Here, $\cL(\Xold) = F_{k}G_{k}F_{k}^{\trans}$ denotes the final Lyapunov
residual of the \(k\)-th Newton step equations.
Observe how the cyclic permutation property of the trace allows turning all
arguments into the final small dense matrices.

Sorting terms by the powers of $\xi$ in~\cref{eqn:poly} leads to five
coefficients of the fourth order polynomial in standard form.
The minimizing argument $\xi_{k}$ is computed from the zeros of
$\frac{d}{d\xi}f_{\cR,k}$.
Then, the actual step size is
\begin{equation*}
  \xi_{k} = \argmin_{\xi\in\Lambda(\tA,\tE)\cap(0,2]} f_{\cR,k}(\xi)
\end{equation*}
for the $3 \times 3$ generalized eigenvalue problem for the matrix pencil
\begin{equation*}
  (\tA,\tE) = \left(
    \begin{bmatrix}
      1&0&0\\
      0&1&0\\
      a_{1}&a_{2}&a_{3}
    \end{bmatrix},
    \begin{bmatrix}
      1&0&0\\
      0&1&0\\
      0&0&a_{4}
    \end{bmatrix}
  \right),
\end{equation*}
where $a = \frac{1}{\nrm{\ha}}\ha$ and $\ha\in\R^{4}$ the with components
\begin{equation*}
  \begin{aligned}
    \ha_{1}&= 2 (v_{4}^{(k)} - v_{1}^{(k)}),\\
    \ha_{2}&= 2 (v_{1}^{(k)} + v_{2}^{(k)} - 2 (v_{4}^{(k)} + v_{5}^{(k)})),\\
    \ha_{3}&= 6 (v_{5}^{(k)} - v_{6}^{(k)}),\\
    \ha_{4}&= 4 v_{3}^{(k)}.\\
  \end{aligned}
\end{equation*}

These last steps are exactly identical to the presentation
in~\cite{BenHSetal16, Wei16}.
Note that additional care is necessary when multiple consecutive iteration steps
use line search since the Riccati residual factors grow with the number of
consecutive line searches and also $\dKnew$ appends a new block of columns,
equal to its own size, with each additional line search iteration.
See the discussion in~\cite{BenHSetal16} after Equation~(5.4) for details.
In our context, the corresponding center matrix $\Dk$ then block diagonally
accumulates the corresponding center matrices rather than simple signed
identities.
Note further, that alternatively an Armijo line search can be used,
but then the step size is limited to the interval $(0,1]$. 

While the line search can help reduce the total number of Newton steps
required, the cost of the single steps can be reduced by an inexact Newton
approach.
The above considerations are ensuring the {\em sufficient decrease condition\/}
\begin{equation*}
  \nrm[F]{\cR(\Xls)} < (1 - \xi_{k} \beta)\nrm[F]{\cR(\Xold)},
\end{equation*}
for a certain positive safety parameter $\beta$.
The inexact Newton acceleration, on the other hand, is controlled by
\begin{equation*}
  \nrm[F]{\cL(\Xnew)} < \tau_{k} \nrm[F]{\cR(\Xold)},
\end{equation*}
for an appropriate forcing sequence ${(\tau_{k})}_{k \in \mathbb{N}}$.
In~\cite{Wei16}, the author suggests $\tau_{k}=\frac{1}{k^{3}+1}$ to achieve
super-linear convergence and $\tau_{k}= \min\{0.1, 0.9\nrm[F]{\cR(\Xold)}\}$ to
preserve quadratic convergence; see~\cite[Table~6.1]{Wei16}.
In general any sequence $\tau_{k} \to 0$ for $k \to \infty$ would
guarantee super-linear convergence, while
$\tau_{k} \in \mathcal{O}(\cR(\Xold))$ ensures
quadratic convergence.
However, note that while the general low-rank inexact Newton framework builds on
the theory in~\cite{FeiHS09}, certain definiteness conditions required in their
central theorem can not be guaranteed in general in the low-rank case such that
the low-rank inexact Newton-Kleinman method may break down.
Implementations need to check this and potentially restart the method without
inexactness.


\subsection{Non-invertible \texorpdfstring{$\boldsymbol{E}$}{E} matrices and
  projected Riccati equations}%
\label{sec:daes}

The examples for CAREs we have considered in \Cref{sec:eqns} are all based on
or associated with linear dynamical systems.
A regularly occurring situation is that these dynamical systems are described
by differential-algebraic rather than ordinary differential equations, in which
case the $E$ matrix in~\cref{eqn:riccati} becomes non-invertible.
Assume that the matrix pencil $\lambda E - A$ is regular, i.e., there exists a
$\lambda \in \C$ such that $\det(\lambda E - A) \neq 0$.
Then, one typically considers the solution of~\cref{eqn:riccati} over the
subspace of finite eigenvalues of $\lambda E - A$ via the projected
Riccati equation
\begin{equation} \label{eqn:projcare}
  \begin{aligned}
    A^{\trans} X E + E^{\trans} X A
      + \cP_{r}^{\trans} C^{\trans} Q C \cP_{r}
      - {\left( B^{\trans} X E + S^{\trans} \cP_{r} \right)}^{\trans} R^{-1}
      \left( B^{\trans} X E + S^{\trans} \cP_{r} \right) & = 0, \\
    \cP_{\ell}^{\trans} X \cP_{\ell} & = X,
  \end{aligned}
\end{equation}
where $\cP_{r}$ and $\cP_{\ell}$ are the right and left projectors
onto the subspace of finite eigenvalues of $\lambda E - A$.
In general, these are given as spectral projectors via the Weierstrass canonical
form of $\lambda E - A$; see, for example,~\cite{BenS14}.
While the necessary computations to obtain these spectral projectors
are typically undesired in the large-scale sparse case, for several practically
occurring matrix structures, the projectors have been formulated explicitly in
terms of parts the coefficient matrices~\cite{Sty08, BenS14}.

In practice, a more efficient approach than explicitly forming $\cP_{r}$ and
$\cP_{\ell}$ is the implicit application of equivalent structural projectors.
In this case, the stabilizing solution of~\cref{eqn:projcare} is directly
computed on the correct lower dimensional subspace.
Similar to the use of the spectral projectors, the implicit projection can, in
practice, only be realized for certain matrix structures, for which the
projectors onto the correct subspaces and truncation of the coefficient matrices
are known by construction;
see, for example,~\cite{BaeBSetal15, FreRM08, SaaV18}.
In all cases, it needs to be noted that the steps in \Cref{alg:newton} do not
change for~\cref{eqn:projcare}.
The case of non-invertible $E$ matrices can typically be implemented by simply
modifying the matrix-matrix and matrix-vector operations needed in
\Cref{alg:newton} to work on the correct subspaces.


\section{Numerical experiments}%
\label{sec:numerics}

The experiments reported here have been executed on a machine with an
AMD Ryzen Threadripper PRO 5975WX 32-Cores processor running at
4.02\,GHz and equipped with 252\,GB total main memory.
The computer is running on Ubuntu 22.04.3 LTS and uses
MATLAB 23.2.0.2365128 (R2023b).
The proposed low-rank Newton-Kleinman method in \Cref{alg:newton} has been
implemented for dense equations using
MORLAB version~6.0~\cite{BenSW23, BenW21c} and for
large-scale sparse equations using M-M.E.S.S.
version~3.0~\cite{BenKS21, SaaKB23}.
The resulting modified versions of these two toolboxes as well as the
source code, data and results of the numerical experiments can be found
at~\cite{supSaaW24a}.
The implementations of \Cref{alg:newton} will be incorporated into the
upcoming releases of M-M.E.S.S. and MORLAB\@.


\subsection{Experimental setup}

An overview about the used example data with the computed equation setups and
corresponding dimensions is shown in \Cref{tab:overview}.
The used example data are:
\begin{description}
  \item[\aircraft{}] is the \texttt{AC10} data set from~\cite{Lei04} modeling
    the linearized vertical-plane dynamics of an aircraft,
  \item[\msd{}] is a mass-spring-damper system with a holonomic constraint
    as described in~\cite{MehS05},
  \item[\rail{(1,6)}] models a heat transfer problem for optimal cooling of
    steel profiles in two differently accurate
    discretizations~\cite{BenS05, morwiki_steel} using the re-implementation~\cite{SaaB20, SaaB21},
  \item[\triplechain{(1,2)}] is the triplechain oscillator benchmark
    introduced in~\cite{TruV09} with two different sets of parameters and
    numbers of masses.
\end{description}
The data set \msd{} has a non-invertible $E$ matrix and is handled via
structured implicit projections as outlined in \Cref{sec:daes}, following
the theory in~\cite{SaaV18}.
To test different scenarios of matrix pencil properties paired with different
weighting terms, we have set up the different formulations of CAREs as motivated
in \Cref{sec:eqns}.
Further on, we denote examples for
equation~\cref{eqn:lqgcare} as \lqg{},
equation~\cref{eqn:hinfcare} as \hinf{},
equation~\cref{eqn:brcare} as \br{}
and equation~\cref{eqn:prcare} as \pr{}.
The modifications of the example data from the literature to fit into the
described equation types can be found in the accompanying code
package~\cite{supSaaW24a}.

To compare the solutions of different computational approaches, we evaluate
three types of scaled residual norms that have been used for similar purposes in
the literature:
\begin{align*}
  \res_{1}(X) & = \frac{\lVert \cR(X) \rVert_{2}}%
    {\lVert \hC^{\trans} \hQ \hC \rVert_{2}}, \\
  \res_{2}(X) & = \frac{\lVert \cR(X) \rVert_{2}}%
    {\lVert \hA \rVert_{2} \lVert E \rVert_{2} \lVert X \rVert_{2}
      + \lVert B R^{-1} B^{\trans} \rVert_{2}}, \\
  \res_{3}(X) & = \frac{\lVert \cR(X) \rVert_{2}}%
    {2 \lVert \hA \rVert_{2} \lVert E \rVert_{2} \lVert X \rVert_{2}
    + \lVert \hC^{\trans} \hQ \hC \rVert_{2}
    + \lVert E \rVert_{2}^{2}  \lVert X \rVert_{2}^{2}
    \lVert B R^{-1} B^{\trans} \rVert_{2}},
\end{align*}
where $\cR(.)$ is the Riccati operator from~\cref{eqn:careoper},
\begin{equation*}
  \hC^{\trans} \hQ \hC = C^{\trans} Q C - S R^{-1} S^{\trans}
  \quad\text{and}\quad
  \hA = A - B R^{-1} S^{\trans}.
\end{equation*}
In the case that multiple algorithms have been used to compute the stabilizing
solution to~\cref{eqn:riccati}, we also compare the relative differences
between these solutions via
\begin{equation*}
  \reldiff(X_{1}, X_{2}) := \frac{\lVert X_{1} - X_{2} \rVert_{2}}%
    {0.5(\lVert X_{1} \rVert_{2} + \lVert X_{2} \rVert_{2})}.
\end{equation*}

\begin{table}[t]
  \centering
  \caption{Overview about example data, matrix dimensions,
    considered equation setups
    and the stability properties of the matrix pencil $\lambda E - A$.
    The first three examples are treated as dense and the latter three as
    large-scale sparse.}%
  \label{tab:overview}
  \vspace{.5\baselineskip}

  \resizebox{\textwidth}{!}{%
  \begin{tabular}{lrrrrrccccl}
    \hline\noalign{\smallskip}
    \multicolumn{1}{c}{\textbf{example}}
      & \multicolumn{1}{c}{$\boldsymbol{n}$}
      & \multicolumn{1}{c}{$\boldsymbol{m}$}
      & \multicolumn{1}{c}{$\boldsymbol{m_{1}}$}
      & \multicolumn{1}{c}{$\boldsymbol{m_{2}}$}
      & \multicolumn{1}{c}{$\boldsymbol{p}$}
      & \multicolumn{1}{c}{\lqg{}}
      & \multicolumn{1}{c}{\hinf{}}
      & \multicolumn{1}{c}{\br{}}
      & \multicolumn{1}{c}{\pr{}}
      & \multicolumn{1}{c}{\textbf{stability}} \\
    \noalign{\smallskip}\hline\noalign{\medskip}
    \aircraft{}
      & $55$
      & $5$
      & $2$
      & $3$
      & $5$
      & \checkmark{}
      & \checkmark{}
      & ---
      & ---
      & unstable \\
    \rail{(1)}
      & $371$
      & $7$
      & $3$
      & $4$
      & $6$
      & \checkmark{}
      & \checkmark{}
      & \checkmark{}
      & \checkmark{}
      & stable \\
    \triplechain{(1)}
      & $602$
      & $1$
      & ---
      & ---
      & $1$
      & ---
      & ---
      & \checkmark{}
      & \checkmark{}
      & stable \\
    \noalign{\medskip}\hline\noalign{\medskip}
    \msd{}
      & $12\,001$
      & $1$
      & ---
      & ---
      & $3$
      & ---
      & ---
      & \checkmark{}
      & ---
      & stable \\
    \triplechain{(2)}
      & $12\,002$
      & $1$
      & ---
      & ---
      & $1$
      & ---
      & ---
      & \checkmark{}
      & \checkmark{}
      & stable \\
    \rail{(6)}
      & $317\,377$
      & $7$
      & $3$
      & $4$
      & $6$
      & \checkmark{}
      & \checkmark{}
      & \checkmark{}
      & \checkmark{}
      & stable \\
    \noalign{\medskip}\hline\noalign{\smallskip}
  \end{tabular}
  }
\end{table}

For compactness of presentation, we introduce the following notation for the
different methods used in the numerical experiments:
\begin{description}
  \item[\newton{}] denotes the Newton-Kleinman method from \Cref{alg:newton},
  \item[\icare{}] is the built-in function from MATLAB for the solution
    of~\cref{eqn:riccati} implementing the algorithm in~\cite{ArnL84},
  \item[\sign{}] denotes the sign function iteration method for Riccati
    equations as described in~\cite{BenEQetal14},
  \item[\ri{}] is the Riccati iteration for the solution of CAREs with
    indefinite quadratic terms; see~\cite{BenHW23}.
\end{description}
Independent of the employed algorithm and the resulting format of the computed
results, e.g., factorized or unfactorized, we denote the final
approximation to the stabilizing solution $X_{\ast}$ by any of the algorithms
as $\Xmax$.


\subsection{Convergence behavior for indefinite terms}%
\label{sec:convexp}

Before we test the proposed method on higher dimensional data sets against
other approaches, we want to investigate the convergence behavior of
\Cref{alg:newton} for the case of indefinite quadratic and constant terms.
In particular the former case is not covered by any convergence theory for
\newton{}.
First, consider the CARE~\cref{eqn:riccati} with the following matrices
\begin{equation} \label{eqn:convergence1}
  \begin{aligned}
    A & = \begin{bmatrix} 2 & 1 \\ 1 & -3 \end{bmatrix}, &
    B & = \begin{bmatrix} 1 & 1 \\ 0 & 2 \end{bmatrix}, &
    R & = \begin{bmatrix} -1 & 0 \\ 0 & 1.5 \end{bmatrix}, &
    C & = \begin{bmatrix} 1 & 1 \end{bmatrix}, \\
    E & = I_{2}, &
    S & = 0, &
    Q & = 1.
  \end{aligned}
\end{equation}
In this example, we have an unstable matrix pencil $\lambda E - A$ with one
eigenvalue in the right open and one eigenvalue in the left open half-plane.
The quadratic weighting term $R$ is indefinite but the constant weighting term
$Q$ is symmetric positive definite.
For the stabilizing solution it holds that $X_{\ast} > 0$ such that besides
\newton{}, \ri{} can be used in this example.
Due to the instability, a stabilizing initial feedback $K_{0}$ is constructed
for \newton{}; see~\cite{supSaaW24a} for details.
The convergence behavior of \newton{} for~\cref{eqn:convergence1} is
shown in \Cref{tab:convergence1}.
We observe that despite the indefinite quadratic term, the iteration provides
quadratic convergence and the intermediate closed-loop matrices
$A_{k} = A - B K_{k}$ are all stable.
However, the monotonic convergence behavior that is theoretically shown for
definite $R$ matrices is clearly not present in this example, since the
eigenvalues of $X_{k} - X_{k-1}$ have different signs for two of the iteration
steps.
Also, the definiteness of $X_{k} - X_{k-1}$ fully changes from step $4$ to $5$.

\begin{table}[t]
  \centering
  \caption{Convergence behavior of the Newton-Kleinman method (\newton{})
    for the example~\cref{eqn:convergence1}:
    For each iteration step, the columns show the normalized residuals, the
    two eigenvalues of the current closed-loop matrix and the two eigenvalues
    of the difference of two consecutive iterates.}%
  \label{tab:convergence1}
  \vspace{.5\baselineskip}

  \begin{tabular}{rrrr}
    \hline\noalign{\smallskip}
    \multicolumn{1}{c}{\textbf{iter.\ step} $\boldsymbol{k}$}
      & \multicolumn{1}{c}{$\boldsymbol{\res_{1}(X_{k})}$}
      & \multicolumn{1}{c}{$\boldsymbol{\Lambda(A_{k})}$}
      & \multicolumn{1}{c}{$\boldsymbol{\Lambda(X_{k} - X_{k - 1})}$} \\
    \noalign{\smallskip}\hline\noalign{\medskip}
    $1$
      & $5.3610\texttt{e-}01$
      & $-1.1049,\,-4.5676$
      & ---\\
    $2$
      & $3.5593\texttt{e-}02$
      & $-1.4100,\,-4.2395$
      & $\phantom{-}3.7386\texttt{e+}00,\,-1.9830\texttt{e-}03$\\
    $3$
      & $6.0872\texttt{e-}05$
      & $-1.4068,\,-4.2451$
      & $-5.0004\texttt{e-}02,\,\phantom{-}2.6109\texttt{e-}05$\\
    $4$
      & $1.5903\texttt{e-}10$
      & $-1.4068,\,-4.2451$
      & $\phantom{-}6.6211\texttt{e-}05,\,\phantom{-}1.1112\texttt{e-}09$\\
    $5$
      & $2.1316\texttt{e-}14$
      & $-1.4068,\,-4.2451$
      & $-1.3313\texttt{e-}10,\,-1.8760\texttt{e-}15$\\
    \noalign{\medskip}\hline\noalign{\smallskip}
  \end{tabular}
\end{table}

As additional verification that \Cref{alg:newton} computes the correct,
stabilizing solution, we compare it to solutions obtained via
\icare{} and \ri{}.
The corresponding residuals are given in the first block of
\Cref{tab:convergence_res} and the relative differences are
\begin{align*}
  \reldiff(\Xmax^{\newton}, \Xmax^{\icare}) & = 8.8968\texttt{e-}15, \\
  \reldiff(\Xmax^{\newton}, \Xmax^{\ri}) & = 1.0286\texttt{e-}13, \\
  \reldiff(\Xmax^{\icare}, \Xmax^{\ri}) & = 1.1175\texttt{e-}13.
\end{align*}
This clearly shows that all methods approximate the same stabilizing solution.

\begin{table}[t]
  \centering
  \caption{Residual norms for all test examples and comparison methods in
    \Cref{sec:convexp}.
    \newton{} provides as accurate or even more accurate solutions compared to
    the standard approach \icare{}.
    \ri{} only works for the first considered scenario and diverges for the
    second one.}%
  \label{tab:convergence_res}
  \vspace{.5\baselineskip}

  \begin{tabular}{llrrr}
    \hline\noalign{\smallskip}
    \multicolumn{1}{c}{\textbf{example}}
      & \multicolumn{1}{c}{\textbf{method}}
      & \multicolumn{1}{c}{$\boldsymbol{\res_{1}(\Xmax)}$}
      & \multicolumn{1}{c}{$\boldsymbol{\res_{2}(\Xmax)}$}
      & \multicolumn{1}{c}{$\boldsymbol{\res_{3}(\Xmax)}$} \\
    \noalign{\smallskip}\hline\noalign{\medskip}
      & \newton{}
      & $9.5151\texttt{e-}15$
      & $2.2825\texttt{e-}16$
      & $8.7894\texttt{e-}18$ \\
    example~\cref{eqn:convergence1}
      & \icare{}
      & $3.5804\texttt{e-}15$
      & $8.5887\texttt{e-}17$
      & $3.3073\texttt{e-}18$ \\
      & \ri{}
      & $3.1979\texttt{e-}12$
      & $7.6713\texttt{e-}14$
      & $2.9540\texttt{e-}15$ \\
    \noalign{\medskip}\hline\noalign{\medskip}
      & \newton{}
      & $1.9453\texttt{e-}14$
      & $3.4368\texttt{e-}16$
      & $1.2723\texttt{e-}17$ \\
    example~\cref{eqn:convergence2}
      & \icare{}
      & $6.0957\texttt{e-}13$
      & $1.0770\texttt{e-}14$
      & $3.9870\texttt{e-}16$ \\
      & \ri{}
      & $1.7227\texttt{e+}40$
      & $2.5740\texttt{e+}19$
      & $8.3380\texttt{e-}02$ \\
    \noalign{\medskip}\hline\noalign{\medskip}
    example~\cref{eqn:convergence3}
      & \newton{}
      & $3.2437\texttt{e-}17$
      & $3.0279\texttt{e-}17$
      & $9.9467\texttt{e-}18$ \\
      & \icare{}
      & $1.9624\texttt{e-}16$
      & $1.8318\texttt{e-}16$
      & $6.0177\texttt{e-}17$ \\
    \noalign{\medskip}\hline\noalign{\smallskip}
  \end{tabular}
\end{table}

Now, we modify the example data by increasing the positive definite part of the
$R$ matrix in~\cref{eqn:convergence1} such that we have now
\begin{equation} \label{eqn:convergence2}
  \begin{aligned}
    A & = \begin{bmatrix} 2 & 1 \\ 1 & -3 \end{bmatrix}, &
    B & = \begin{bmatrix} 1 & 1 \\ 0 & 2 \end{bmatrix}, &
    R & = \begin{bmatrix} -1 & 0 \\ 0 & 2 \end{bmatrix}, &
    C & = \begin{bmatrix} 1 & 1 \end{bmatrix}, \\
    E & = I_{2}, &
    S & = 0, &
    Q & = 1.
  \end{aligned}
\end{equation}
Similar to~\cref{eqn:convergence1}, we consider the case of an indefinite
weighing matrix in the quadratic term of~\cref{eqn:riccati}; however, the change
in the data results in the stabilizing solution $X_{\ast}$ being indefinite.
The convergence behavior of \newton{} for this case is shown in
\Cref{tab:convergence2}.
As in the previous example, the convergence is quadratic towards the
stabilizing solution and the iterates do not show any monotonicity.
Additionally, we do not have the stability of all closed-loop matrices during
the iteration as the one computed in the first step is clearly unstable.
We do not expect \ri{} to work for this case due to $X_{\ast}$ being indefinite
and, in fact, we see in the second block row of \Cref{tab:convergence_res} that
\ri{} does not converge to a solution of~\cref{eqn:riccati}.
However, \newton{} clearly converges to the correct solution with a relative
difference to the solution computed by \icare{} of
\begin{equation*}
  \reldiff(\Xmax^{\newton{}}, \Xmax^{\icare{}}) = 5.6279\texttt{e-}15.
\end{equation*}

\begin{table}[t]
  \centering
  \caption{Convergence behavior of the Newton-Kleinman method (\newton{})
    for the example~\cref{eqn:convergence2}:
    For each iteration step, the columns show the normalized residuals, the
    two eigenvalues of the current closed-loop matrix and the two eigenvalues
    of the difference of two consecutive iterates.}%
  \label{tab:convergence2}
  \vspace{.5\baselineskip}

  \begin{tabular}{rrrr}
    \hline\noalign{\smallskip}
    \multicolumn{1}{c}{\textbf{iter.\ step} $\boldsymbol{k}$}
      & \multicolumn{1}{c}{$\boldsymbol{\res_{1}(X_{k})}$}
      & \multicolumn{1}{c}{$\boldsymbol{\Lambda(A_{k})}$}
      & \multicolumn{1}{c}{$\boldsymbol{\Lambda(X_{k} - X_{k - 1})}$} \\
    \noalign{\smallskip}\hline\noalign{\medskip}
    $1$
      & $1.3423\texttt{e+}01$
      & $\phantom{-}2.3071,\,-7.8315$
      & ---\\
    $2$
      & $3.1646\texttt{e-}01$
      & $-4.1113,\,-1.4164$
      & $-1.3696\texttt{e+}02,\,-7.4641\texttt{e-}04$\\
    $3$
      & $8.2620\texttt{e-}03$
      & $-4.0451,\,-1.4620$
      & $-7.8472\texttt{e-}01,\,\phantom{-}2.3569\texttt{e-}04 $\\
    $4$
      & $1.9458\texttt{e-}06$
      & $-4.0448,\,-1.4626$
      & $-8.1348\texttt{e-}03,\,\phantom{-}7.4808\texttt{e-}08$ \\
    $5$
      & $3.3469\texttt{e-}14$
      & $-4.0448,\,-1.4626$
      & $\phantom{-}1.5635\texttt{e-}06,\,\phantom{-}1.1925\texttt{e-}11$ \\
    \noalign{\medskip}\hline\noalign{\smallskip}
  \end{tabular}
\end{table}

As final preliminary example, we want to investigate the effect of an
indefinite constant term.
Therefore, we modify the previous example as follows
\begin{equation} \label{eqn:convergence3}
  \begin{aligned}
    A & = \begin{bmatrix} 2 & 1 \\ 1 & -3 \end{bmatrix}, &
    B & = \begin{bmatrix} 1 \\ 1 \end{bmatrix}, &
    Q & = \begin{bmatrix} 1 & 0 \\ 0 & -2 \end{bmatrix}, &
    C & = \begin{bmatrix} 1 & 1 \\ 0 & 2 \end{bmatrix}, \\
    E & = I_{2}, &
    S & = 0, &
    R & = 1.
  \end{aligned}
\end{equation}
Since we have already seen the effects of an indefinite quadratic term, we
consider here $R > 0$ for simplicity.
The stabilizing solution in this example is indefinite again.
\Cref{tab:convergence3} shows the convergence behavior of \newton{} for this
example case.
We see exactly what was expected from \Cref{thm:convergence}:
the closed-loop matrices are stable in all steps, the convergence is quadratic
and monotonic.
Since \ri{} has not been extended to the case of indefinite constant terms and
the solution is not positive semi-definite, we omit the comparing computations
with this method here and only provide the results of \icare{} instead.
The residual norms can be found in the third block row of
\Cref{tab:convergence_res} and the relative difference between the solutions
computed by \newton{} and \icare{} is
\begin{equation*}
  \reldiff(\Xmax^{\newton{}}, \Xmax^{\icare{}}) = 8.8968\texttt{e-}15.
\end{equation*}
Both methods appear to approximate the same stabilizing solution.

\begin{table}[t]
  \centering
  \caption{Convergence behavior of the Newton-Kleinman method (\newton{})
    for the example~\cref{eqn:convergence3}:
    For each iteration step, the columns show the normalized residuals, the
    two eigenvalues of the current closed-loop matrix and the two eigenvalues
    of the difference of two consecutive iterates.}%
  \label{tab:convergence3}
  \vspace{.5\baselineskip}

  \begin{tabular}{rrrr}
    \hline\noalign{\smallskip}
    \multicolumn{1}{c}{\textbf{iter.\ step} $\boldsymbol{k}$}
      & \multicolumn{1}{c}{$\boldsymbol{\res_{1}(X_{k})}$}
      & \multicolumn{1}{c}{$\boldsymbol{\Lambda(A_{k})}$}
      & \multicolumn{1}{c}{$\boldsymbol{\Lambda(X_{k} - X_{k - 1})}$} \\
    \noalign{\smallskip}\hline\noalign{\medskip}
    $1$
      & $1.9109\texttt{e-}01$
      & $-3.3289,\,-0.8111$
      & ---\\
    $2$
      & $1.4573\texttt{e-}02$
      & $-3.2914,\,-0.5227$
      & $-3.3630\texttt{e-}01,\,-1.3069\texttt{e-}02$ \\
    $3$
      & $7.0984\texttt{e-}04$
      & $-3.2887,\,-0.4383$
      & $-5.9143\texttt{e-}02,\,-9.6054\texttt{e-}04$ \\
    $4$
      & $5.7445\texttt{e-}06$
      & $-3.2886,\,-0.4301$
      & $-5.0185\texttt{e-}03,\,-7.8659\texttt{e-}06$ \\
    $5$
      & $5.0562\texttt{e-}10$
      & $-3.2886,\,-0.4300$
      & $-4.6517\texttt{e-}05,\,-4.3494\texttt{e-}09$ \\
    $6$
      & $3.0792\texttt{e-}17$
      & $-3.2886,\,-0.4300$
      & $-4.2059\texttt{e-}09,\,-2.1477\texttt{e-}14$ \\
    \noalign{\medskip}\hline\noalign{\smallskip}
  \end{tabular}
\end{table}


\subsection{Numerical comparisons}

In this section, we compare the proposed algorithm with established solvers
on different benchmark data sets from the literature and equation
scenarios.
While we concentrate on examples with small-scale dense coefficient
matrices in the first part to establish trust into the proposed \newton{}
method, we present results for large-scale sparse matrices afterwards.
The inexact Newton-Kleinman method with line search has been implemented
for the large-scale sparse case.
However, we have observed some inconsistent behavior concerning the considered
example setups due to which we decided to present only the results for the exact
Newton-Kleinman method.


\subsubsection{Examples with dense coefficient matrices}%
\label{sec:denseexp}

\begin{table}
  \centering
  \caption{Residual norms for all dense test examples and
    comparison methods in \Cref{sec:denseexp}.
    \newton{} provides reasonably accurate and often the most accurate
    approximations compared to \sign{} and \icare{}.}%
  \label{tab:dense_res}
  \vspace{.5\baselineskip}

  \begin{tabular}{llrrr}
    \hline\noalign{\smallskip}
    \multicolumn{1}{c}{\textbf{example}}
      & \multicolumn{1}{c}{\textbf{method}}
      & \multicolumn{1}{c}{$\boldsymbol{\res_{1}(\Xmax)}$}
      & \multicolumn{1}{c}{$\boldsymbol{\res_{2}(\Xmax)}$}
      & \multicolumn{1}{c}{$\boldsymbol{\res_{3}(\Xmax)}$} \\
    \noalign{\smallskip}\hline\noalign{\medskip}
      & \newton{}
      & $4.1073\texttt{e-}08$
      & $1.3682\texttt{e-}20$
      & $4.5781\texttt{e-}27$ \\
    \aircraft\hfill(\lqg)
      & \sign{}
      & $2.6483\texttt{e-}05$
      & $8.8220\texttt{e-}18$
      & $2.9519\texttt{e-}24$ \\
      & \icare{}
      & $1.0713\texttt{e-}07$
      & $3.5686\texttt{e-}20$
      & $1.1941\texttt{e-}26$ \\
    \noalign{\medskip}\hline\noalign{\medskip}
      & \newton{}
      & $7.4736\texttt{e-}07$
      & $2.4035\texttt{e-}20$
      & $9.7826\texttt{e-}26$ \\
    \aircraft\hfill(\hinf)
      & \sign{}
      & $3.0047\texttt{e-}06$
      & $9.6631\texttt{e-}20$
      & $3.9330\texttt{e-}25$ \\
      & \icare{}
      & $1.9949\texttt{e-}05$
      & $6.4155\texttt{e-}19$
      & $2.6112\texttt{e-}24$ \\
    \noalign{\medskip}\hline\noalign{\medskip}
      & \newton{}
      & $7.5678\texttt{e-}14$
      & $1.0229\texttt{e-}14$
      & $8.8312\texttt{e-}16$ \\
    \rail{(1)}\hfill(\lqg)
      & \sign{}
      & $3.1905\texttt{e-}10$
      & $4.3122\texttt{e-}11$
      & $3.7231\texttt{e-}12$ \\
      & \icare{}
      & $2.0819\texttt{e-}13$
      & $2.8139\texttt{e-}14$
      & $2.4294\texttt{e-}15$ \\
    \noalign{\medskip}\hline\noalign{\medskip}
      & \newton{}
      & $2.4749\texttt{e-}12$
      & $4.0514\texttt{e-}13$
      & $8.6384\texttt{e-}16$ \\
    \rail{(1)}\hfill(\hinf)
      & \sign{}
      & $5.5336\texttt{e-}10$
      & $9.0584\texttt{e-}11$
      & $1.9314\texttt{e-}13$ \\
      & \icare{}
      & $5.4824\texttt{e-}14$
      & $8.9746\texttt{e-}15$
      & $1.9136\texttt{e-}17$ \\
    \noalign{\medskip}\hline\noalign{\medskip}
      & \newton{}
      & $1.5986\texttt{e-}13$
      & $1.4855\texttt{e-}14$
      & $6.2186\texttt{e-}15$ \\
    \rail{(1)}\hfill(\br)
      & \sign{}
      & $3.1304\texttt{e-}14$
      & $2.9089\texttt{e-}15$
      & $1.2178\texttt{e-}15$ \\
      & \icare{}
      & $1.3867\texttt{e-}13$
      & $1.2886\texttt{e-}14$
      & $5.3943\texttt{e-}15$ \\
    \noalign{\medskip}\hline\noalign{\medskip}
      & \newton{}
      & $9.3161\texttt{e-}12$
      & $6.4179\texttt{e-}13$
      & $2.3614\texttt{e-}13$ \\
    \rail{(1)}\hfill(\pr)
      & \sign{}
      & $4.3343\texttt{e-}10$
      & $2.9859\texttt{e-}11$
      & $1.0986\texttt{e-}11$ \\
      & \icare{}
      & $6.4693\texttt{e-}14$
      & $4.4567\texttt{e-}15$
      & $1.6398\texttt{e-}15$ \\
    \noalign{\medskip}\hline\noalign{\medskip}
      & \newton{}
      & $3.6923\texttt{e-}11$
      & $3.0799\texttt{e-}16$
      & $1.5397\texttt{e-}16$ \\
    \triplechain{(1)}\hfill(\br)
      & \sign{}
      & $9.2343\texttt{e-}11$
      & $7.7027\texttt{e-}16$
      & $3.8506\texttt{e-}16$ \\
      & \icare{}
      & $1.6221\texttt{e-}10$
      & $1.3531\texttt{e-}15$
      & $6.7641\texttt{e-}16$ \\
    \noalign{\medskip}\hline\noalign{\medskip}
      & \newton{}
      & $6.4378\texttt{e-}12$
      & $6.9411\texttt{e-}17$
      & $1.4807\texttt{e-}15$ \\
    \triplechain{(1)}\hfill(\pr)
      & \sign{}
      & $3.6829\texttt{e-}12$
      & $3.9708\texttt{e-}17$
      & $8.4708\texttt{e-}16$ \\
      & \icare{}
      & $1.6286\texttt{e-}11$
      & $1.7559\texttt{e-}16$
      & $3.7457\texttt{e-}15$ \\
    \noalign{\medskip}\hline\noalign{\smallskip}
  \end{tabular}
\end{table}

\begin{table}[t]
  \centering
  \caption{Relative differences for the dense test examples in
    \Cref{sec:denseexp}.
    All differences are reasonably low such that numerically we can rely
    on the results provided by \newton{}.}%
  \label{tab:dense_reldiff}
  \vspace{.5\baselineskip}

  \begin{tabular}{lrr}
    \hline\noalign{\smallskip}
    \multicolumn{1}{c}{\textbf{example}}
      & $\boldsymbol{\reldiff(\Xmax^{\newton}, \Xmax^{\sign})}$
      & $\boldsymbol{\reldiff(\Xmax^{\newton}, \Xmax^{\icare})}$ \\
    \noalign{\smallskip}\hline\noalign{\medskip}
    \aircraft\hfill(\lqg)
      & $1.5687\texttt{e-}12$
      & $5.2915\texttt{e-}14$ \\
    \aircraft\hfill(\hinf)
      & $3.4874\texttt{e-}13$
      & $5.4917\texttt{e-}13$ \\
    \rail{(1)}\hfill(\lqg)
      & $3.3423\texttt{e-}10$
      & $1.0539\texttt{e-}11$ \\
    \rail{(1)}\hfill(\hinf)
      & $8.0362\texttt{e-}10$
      & $8.0170\texttt{e-}10$ \\
    \rail{(1)}\hfill(\br)
      & $9.2682\texttt{e-}13$
      & $9.2764\texttt{e-}13$ \\
    \rail{(1)}\hfill(\pr)
      & $6.4985\texttt{e-}10$
      & $6.6590\texttt{e-}10$ \\
    \triplechain{(1)}\hfill(\br)
      & $2.3524\texttt{e-}11$
      & $4.8920\texttt{e-}11$ \\
    \triplechain{(1)}\hfill(\pr)
      & $1.7659\texttt{e-}12$
      & $1.1858\texttt{e-}10$ \\
    \noalign{\medskip}\hline\noalign{\smallskip}
  \end{tabular}
\end{table}

An overview about the experiments presented in this section is given in
the first block row of \Cref{tab:overview}.
We decided to start by experimenting with small-scale dense coefficient matrices since
for this case, there are well established solvers that can handle the general
case~\cref{eqn:riccati}, which we consider in this paper.
Such a variety of methods is not given for large-scale sparse matrices;
therefore, here, we numerically establish trust into the solutions
obtained by \newton{} and show that they provide reasonable accuracy in
comparison to other approaches.
As inner solver for the occurring Lyapunov equations, we use the
$LDL^{\trans}$-factorized sign function iteration
method~\cite[Alg.~7]{BenW21b} from the MORLAB toolbox~\cite{BenSW23, BenW21c}.
For the comparison, we have selected \sign{} and \icare{} as two
well-established approaches for general CAREs with dense coefficient matrices.
The results of the experiments are shown in \Cref{tab:dense_res} in form of the
residual norms for the different methods and in \Cref{tab:dense_reldiff}, which
shows the relative differences between the solutions computed by the different
approaches.
For further experimental metrics such as the amount of iteration steps taken
by \newton{} and \sign{}, computation times, and more, we refer the reader to
the log files of the experiments in the accompanying code
package~\cite{supSaaW24a}.

Overall, we can evaluate that \newton{} performs comparably well or even best
among all those methods.
Note that we used $10^{-12}$ as convergence tolerance for the normalized
residual norm internally computed by \newton{} such that we do not expect much
smaller values for $\res_{1}(\Xmax)$ in \Cref{tab:dense_res}.
Despite that, \newton{} shows in various examples up to one order of magnitude
better residuals than \icare{} and often several orders of magnitude better
residuals than \sign{}.
The relative differences in \Cref{tab:dense_reldiff} show numerically that all
three methods approximate the same stabilizing solution to the example
equations and provide similar solutions with many significant digits of
accuracy in common.
With these results at hand, we believe that applying \newton{} in the
large-scale sparse setting will provide correct as well as sufficiently accurate
solutions.


\subsubsection{Examples with large-scale sparse coefficient matrices}%
\label{sec:sparseexp}

\begin{table}[t]
  \centering
  \caption{Residual norms for all sparse test examples and
    comparison methods in \Cref{sec:sparseexp}.
    \newton{} provides very accurate approximations throughout all examples
    with residual norms up to eight orders of magnitude smaller than \ri{}.}%
  \label{tab:sparse_res}
  \vspace{.5\baselineskip}

  \begin{tabular}{llrrr}
    \hline\noalign{\smallskip}
    \multicolumn{1}{c}{\textbf{example}}
      & \multicolumn{1}{c}{\textbf{method}}
      & \multicolumn{1}{c}{$\boldsymbol{\res_{1}(\Xmax)}$}
      & \multicolumn{1}{c}{$\boldsymbol{\res_{2}(\Xmax)}$}
      & \multicolumn{1}{c}{$\boldsymbol{\res_{3}(\Xmax)}$} \\
    \noalign{\smallskip}\hline\noalign{\medskip}
    \msd\hfill(\br)
      & \newton{}
      & $2.1781\texttt{e-}13$
      & $1.4945\texttt{e-}17$
      & $1.8157\texttt{e-}19$ \\
      & \ri{}
      & $1.6943\texttt{e-}08$
      & $1.1626\texttt{e-}12$
      & $1.4124\texttt{e-}14$ \\
    \noalign{\medskip}\hline\noalign{\medskip}
    \triplechain{(2)}\hfill(\br)
      & \newton{}
      & $7.0476\texttt{e-}13$
      & $7.5953\texttt{e-}21$
      & $3.4138\texttt{e-}21$ \\
      & \ri{}
      & $3.6606\texttt{e-}05$
      & $3.9451\texttt{e-}13$
      & $1.7732\texttt{e-}13$ \\
    \noalign{\medskip}\hline\noalign{\medskip}
    \triplechain{(2)}\hfill(\pr)
      & \newton{}
      & $5.0010\texttt{e-}13$
      & $4.6881\texttt{e-}21$
      & $3.5724\texttt{e-}24$ \\
      & \ri{}
      & $3.4240\texttt{e-}05$
      & $3.2098\texttt{e-}13$
      & $2.4459\texttt{e-}16$ \\
    \noalign{\medskip}\hline\noalign{\medskip}
    \rail{(6)}\hfill(\lqg)
      & \newton{}
      & $9.6445\texttt{e-}13$
      & $6.6215\texttt{e-}14$
      & $2.8035\texttt{e-}14$ \\
    \noalign{\medskip}\hline\noalign{\medskip}
    \rail{(6)}\hfill(\hinf)
      & \newton{}
      & $4.5192\texttt{e-}13$
      & $2.8799\texttt{e-}14$
      & $1.4346\texttt{e-}15$ \\
      & \ri{}
      & $1.1576\texttt{e-}10$
      & $7.3768\texttt{e-}12$
      & $3.6748\texttt{e-}13$ \\
    \noalign{\medskip}\hline\noalign{\medskip}
    \rail{(6)}\hfill(\br)
      & \newton{}
      & $8.8948\texttt{e-}14$
      & $4.6198\texttt{e-}15$
      & $2.2423\texttt{e-}15$ \\
      & \ri{}
      & $1.6442\texttt{e-}10$
      & $8.5396\texttt{e-}12$
      & $4.1448\texttt{e-}12$ \\
    \noalign{\medskip}\hline\noalign{\medskip}
    \rail{(6)}\hfill(\pr)
      & \newton{}
      & $2.8870\texttt{e-}13$
      & $5.3150\texttt{e-}16$
      & $2.6375\texttt{e-}16$
      \\
      & \ri{}
      & $5.7022\texttt{e-}09$
      & $1.0498\texttt{e-}11$
      & $5.2095\texttt{e-}12$
      \\
    \noalign{\medskip}\hline\noalign{\smallskip}
  \end{tabular}
\end{table}

\begin{table}[t]
  \centering
  \caption{Relative differences for all sparse examples in \Cref{sec:sparseexp}.
    All differences are reasonably low such that numerically we can rely
    on the results provided by \newton{}.
    Due to the lack of comparison methods, relative differences could not be
    provided for all test scenarios.}%
  \label{tab:sparse_reldiff}
  \vspace{.5\baselineskip}

  \begin{tabular}{lr}
    \hline\noalign{\smallskip}
    \multicolumn{1}{c}{\textbf{example}}
      & $\boldsymbol{\reldiff(\Xmax^{\newton}, \Xmax^{\ri})}$ \\
    \noalign{\smallskip}\hline\noalign{\medskip}
    \msd\hfill(\br) & $1.8490\texttt{e-}12$ \\
    \triplechain{(2)}\hfill(\br) & $6.9416\texttt{e-}11$ \\
    \triplechain{(2)}\hfill(\pr) & $6.9221\texttt{e-}11$\\
    \rail{(6)}\hfill(\hinf) & $5.4646\texttt{e-}10$ \\
    \rail{(6)}\hfill(\br) & $4.9422\texttt{e-}08$ \\
    \rail{(6)}\hfill(\pr) & $8.6206\texttt{e-}10$\\
    \noalign{\medskip}\hline\noalign{\smallskip}
  \end{tabular}
\end{table}

\begin{table}[t]
  \centering
  \caption{Numbers of performed iteration steps for all sparse test examples and
    comparison methods in \Cref{sec:sparseexp}.}%
  \label{tab:sparse_iter}
  \vspace{.5\baselineskip}

  \begin{tabular}{lrr}
    \hline\noalign{\smallskip}
    \multicolumn{1}{c}{\textbf{example}}
      & \multicolumn{1}{c}{\textbf{\# \newton{} iteration steps}}
      & \multicolumn{1}{c}{\textbf{\# \ri{} iteration steps}} \\
    \noalign{\smallskip}\hline\noalign{\medskip}
    \msd\hfill(\br)
      & $6$
      & $4$ \\
    \triplechain{(2)}\hfill(\br)
      & $2$
      & $2$ \\
    \triplechain{(2)}\hfill(\pr)
      & $4$
      & $4$ \\
    \rail{(6)}\hfill(\lqg)
      & $9$
      & --- \\
    \rail{(6)}\hfill(\hinf)
      & $11$
      & $3$ \\
    \rail{(6)}\hfill(\br)
      & $6$
      & $3$ \\
    \rail{(6)}\hfill(\pr)
      & $7$
      & $6$ \\
    \noalign{\medskip}\hline\noalign{\smallskip}
  \end{tabular}
\end{table}

Now we consider the case of CARE examples with large-scale sparse coefficient
matrices.
An overview about these experiments is given in the second block row of
\Cref{tab:overview}.
Whenever possible, we used \ri{} as comparison method where we chose RADI as
solver for the Riccati equations with positive semi-definite quadratic terms
occurring in each step of the iteration.
For the implementation of \newton{}, we use the $LDL^{\trans}$-factorized
ADI method~\cite{LanMS14} as the solver of the inner Lyapunov equations.
The iterations are stopped when the implicitly computed $\res_{1}$ are below the
convergence tolerance $10^{-12}$.
In \ri{} on the other hand, we use the RADI method~\cite{BenBKetal18} to solve
the occurring classical Riccati equations as efficient as possible.
The residual norms of the computed results are shown in \Cref{tab:sparse_res},
the relative differences for examples in which \newton{} and \ri{} could be
applied can be found in \Cref{tab:sparse_reldiff}
and \Cref{tab:sparse_iter} shows the number of performed iteration steps for
the two compared methods.

The residual norms in \Cref{tab:sparse_res} show \newton{} to provide accurate
solutions to all example equations.
It stands out that, in all examples, \newton{} provides residual norms that are
at least three orders of magnitude better than those of the solutions provided
by \ri{}.
One possible explanation for these results is that in \ri{}, the overall
solution is accumulated via column concatenation and truncation.
This easily leads to the loss of numerical accuracy especially in the cases
when the stabilizing solution is badly conditioned.
For \rail{(6)}~(\lqg), we could not use \ri{} for the comparison, since the
constant term in this example is indefinite by construction.
The stabilizing solution however is numerically positive semi-definite.

\begin{figure}[t]
  \centering
  \begin{subfigure}[b]{.49\textwidth}
    \centering
  \tikzexternalenable%
  \tikzsetnextfilename{rail_lqg}%
  \begin{tikzpicture}
  \pgfplotstableread{graphics/data/large_rail_nm_lqg.dat}\tableNEWTON

  \begin{semilogyaxis}[%
    width  = .675\textwidth,
    height = .4\textwidth,
    scale only axis,
    xmin = 0,
    xmax = 40,
    ymin = 1e-14,
    ymax = 1e+2,
    xminorticks = false,
    yminorticks = false,
    xlabel = {computation time (min)},
    ylabel = {implicit residual},
    ylabel style   = {yshift = -.3em},
    scaled x ticks = false,
    x tick label style = {/pgf/number format/fixed},
    clip mode          = individual]

    \addplot[newton] table[x index = 0, y index = 1] {\tableNEWTON};
  \end{semilogyaxis}
\end{tikzpicture}%
  \tikzexternaldisable%

    \caption{\rail{(6)}~(\lqg).}\label{fig:rail_lqg}
  \end{subfigure}%
  \hfill%
  \begin{subfigure}[b]{.49\textwidth}
    \centering
  \tikzexternalenable%
  \tikzsetnextfilename{rail_hinf}%
  \begin{tikzpicture}
  \pgfplotstableread{graphics/data/large_rail_nm_hinf.dat}\tableNEWTON
  \pgfplotstableread{graphics/data/large_rail_ri_hinf.dat}\tableRI

  \begin{semilogyaxis}[%
    width  = .675\textwidth,
    height = .4\textwidth,
    scale only axis,
    xmin = 0,
    xmax = 48,
    ymin = 1e-22,
    ymax = 1e+4,
    xminorticks = false,
    yminorticks = false,
    xlabel = {computation time (min)},
    ylabel = {implicit residual},
    ylabel style   = {yshift = -.3em},
    scaled x ticks = false,
    x tick label style = {/pgf/number format/fixed},
    clip mode          = individual]

    \addplot[newton] table[x index = 0, y index = 1] {\tableNEWTON};
    \addplot[ri] table[x index = 0, y index = 1] {\tableRI};
  \end{semilogyaxis}
\end{tikzpicture}%
  \tikzexternaldisable%

    \caption{\rail{(6)}~(\hinf).}\label{fig:rail_hinf}
  \end{subfigure}

  \vspace{.5\baselineskip}
  \begin{subfigure}[b]{.49\textwidth}
    \centering
  \tikzexternalenable%
  \tikzsetnextfilename{rail_br}%
  \begin{tikzpicture}
  \pgfplotstableread{graphics/data/large_rail_nm_br.dat}\tableNEWTON
  \pgfplotstableread{graphics/data/large_rail_ri_br.dat}\tableRI

  \begin{semilogyaxis}[%
    width  = .675\textwidth,
    height = .4\textwidth,
    scale only axis,
    xmin = 0,
    xmax = 14,
    ymin = 1e-15,
    ymax = 1e+1,
    xminorticks = false,
    yminorticks = false,
    xlabel = {computation time (min)},
    ylabel = {implicit residual},
    ylabel style   = {yshift = -.3em},
    scaled x ticks = false,
    x tick label style = {/pgf/number format/fixed},
    clip mode          = individual]

    \addplot[newton] table[x index = 0, y index = 1] {\tableNEWTON};
    \addplot[ri] table[x index = 0, y index = 1] {\tableRI};
  \end{semilogyaxis}
\end{tikzpicture}%
  \tikzexternaldisable%

    \caption{\rail{(6)}~(\br).}\label{fig:rail_br}
  \end{subfigure}%
  \hfill%
  \begin{subfigure}[b]{.49\textwidth}
    \centering
  \tikzexternalenable%
  \tikzsetnextfilename{rail_pr}%
  \begin{tikzpicture}
  \pgfplotstableread{graphics/data/large_rail_nm_pr.dat}\tableNEWTON
  \pgfplotstableread{graphics/data/large_rail_ri_pr.dat}\tableRI

  \begin{semilogyaxis}[%
    width  = .675\textwidth,
    height = .4\textwidth,
    scale only axis,
    xmin = 0,
    xmax = 28,
    ymin = 1e-14,
    ymax = 1e+8,
    xminorticks = false,
    yminorticks = false,
    xlabel = {computation time (min)},
    ylabel = {implicit residual},
    ylabel style   = {yshift = -.3em},
    scaled x ticks = false,
    x tick label style = {/pgf/number format/fixed},
    clip mode          = individual]

    \addplot[newton] table[x index = 0, y index = 1] {\tableNEWTON};
    \addplot[ri] table[x index = 0, y index = 1] {\tableRI};
  \end{semilogyaxis}
\end{tikzpicture}%
  \tikzexternaldisable%

    \caption{\rail{(6)}~(\pr).}\label{fig:rail_rp}
  \end{subfigure}

  \vspace{.5\baselineskip}
  \tikzexternalenable%
  \tikzsetnextfilename{rail_legend}%
  \begin{tikzpicture}
  \begin{axis}[%
    hide axis,
    width  = 1mm,
    height = 1mm,
    scale only axis,
    xmin = 0,
    xmax = 1,
    ymin = 0,
    ymax = 1,
    legend columns = 4,
    legend style   = {
      at     = {(0,0)},
      anchor = center,
      /tikz/every even column/.append style = {column sep = 0.2cm}},
    legend cell align  = {left},
    clip mode          = individual]

    \addlegendimage{newton}
    \addlegendentry{\newton{}}

    \addlegendimage{ri}
    \addlegendentry{\ri{}}
  \end{axis}
\end{tikzpicture}%
  \tikzexternaldisable%

  \caption{Convergence of \newton{} and \ri{} for all example equations with the
    \rail{(6)} data set:
    We can see that in all examples where it was applicable \ri{} obtains its
    final approximation significantly faster than \newton{}.
    This comes from the use of RADI as underlying solver.
    However, the shown implicit residual computed by \ri{} is not accurate as
    shown in \Cref{tab:sparse_res}.}%
  \label{fig:rail}
\end{figure}
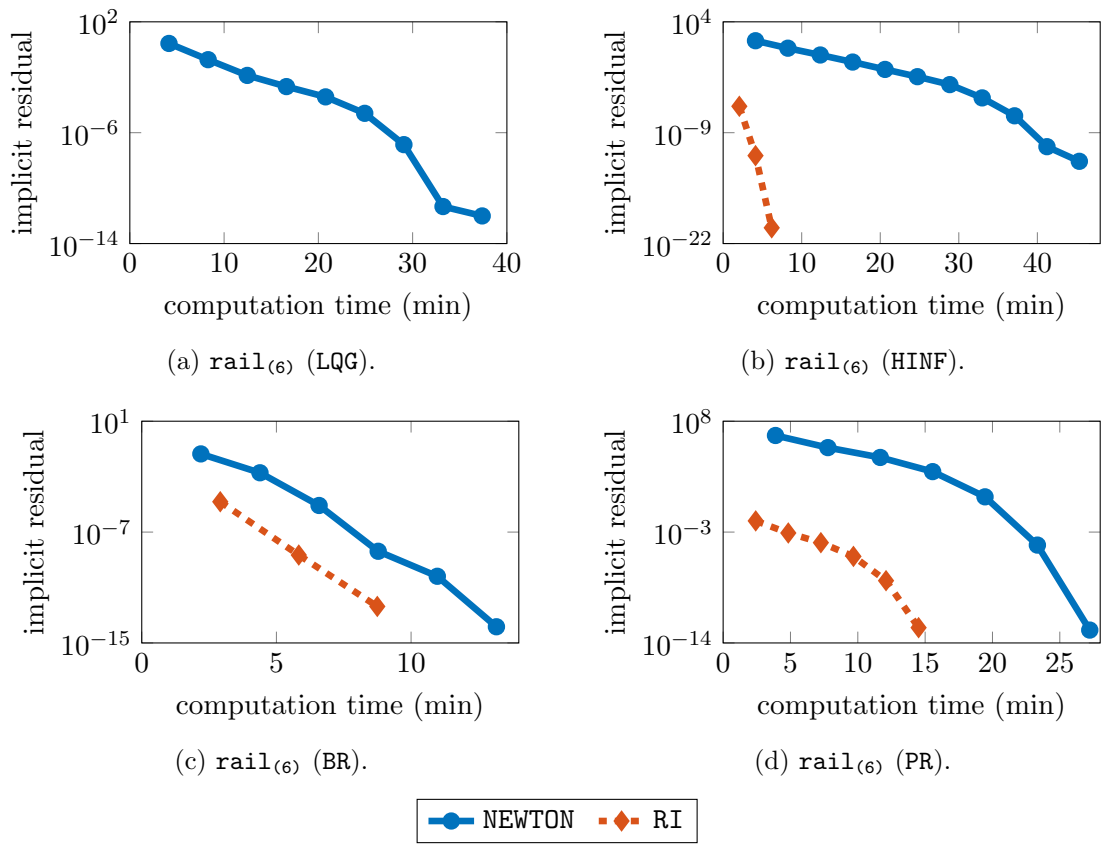

The convergence behavior of \newton{} and \ri{} for the example equations on
the data set \rail{(6)} is illustrated in \Cref{fig:rail}.
These and similar plots for the other sparse examples can be found in the
accompanying code package~\cite{supSaaW24a}.
The plots show that for the cases that \ri{} was applicable, it strongly
outperformed \newton{} in terms of computation time.
This is a result of the choice for the internal CARE solver in \ri{}, which in
our experiments was the RADI method~\cite{BenBKetal18}.
The residuals shown are those that the methods implicitly compute during the
iterations to determine convergence.
Comparing these plots with \Cref{tab:sparse_res} reveals that the residuals
internally computed by \ri{} strongly diverge from the actual normalized
residual norm $\res_{1}(\Xmax)$, which is several orders of magnitude larger.
On the other hand, for \newton{} the results seem to coincide very well.


\section{Conclusions}%
\label{sec:conclusions}

In this work, we presented a new formulation of the Newton-Kleinman iteration
for solving general continuous-time algebraic Riccati equations with large-scale
sparse coefficient matrices using low-rank indefinite symmetric $LDL^{\trans}$
factorizations of the solution.
For relevant scenarios from the literature, we could show the theoretical
convergence of the algorithm.
We provided the updated formulas for an exact line search procedure and inexact
inner solves, and we outlined how to handle the case of projected algebraic
Riccati equations occurring for matrix pencils with infinite eigenvalues.
The numerical experiments show that our proposed algorithm provides reliable and
accurate solutions to the considered problem and that even in the cases for
which we could not provide a convergence theory, the algorithm appears to work
perfectly fine.

While we were able to provide convergence results for many of the practically
occurring cases, the convergence behavior for the case of indefinite quadratic
terms remains unsolved.
The numerical results suggest that even in this situation, the proposed
Newton-Kleinman method converges to the correct solution, however, the lack of
monotonicity in the constructed iterates prevents the use of established
strategies for proving convergence.
Also, we have observed in our experiments that, while the new Newton-Kleinman
iteration outperformed all comparing methods (if there were any at all) in
terms of accuracy, it could not compete in the large-scale sparse case with the
computational speed of the Riccati iteration that employed the RADI method
as inner solver.
Therefore, it is in our interest to investigate possible extensions of other,
potentially faster performing methods to the case of general algebraic Riccati equations.


\clearpage%
\addcontentsline{toc}{section}{References}
\bibliographystyle{plainurl}
\bibliography{bibtex/myref}

\end{document}